\documentclass[12pt]{article}
\pdfoutput=1
\usepackage{lineno}
\usepackage{graphicx}
\usepackage{mathrsfs}
\usepackage{amssymb}
\usepackage{amsmath}
\usepackage{graphicx}
\usepackage{amsmath,amsfonts}
\usepackage{algorithm}
\usepackage[square]{natbib}
\usepackage{amscd}
\usepackage[mathscr]{eucal}
\usepackage{lineno}
\usepackage{extarrows}
\usepackage{etoolbox}
\usepackage{tikz}
\usepackage{url}
\usepackage{listings}
\usepackage{bm}
\usepackage{xr}
\usepackage{url}
\usepackage{comment}
\usepackage{algpseudocode}
\makeatletter
\newcommand{\labeltext}[2]{%
    \@bsphack
    \csname phantomsection\endcsname 
     \def\@currentlabel{#1}{\label{#2}}%
     \@esphack
}

\makeatletter
\patchcmd{\NAT@test}{\else \NAT@nm}{\else \NAT@nmfmt{\NAT@nm}}{}{}

\DeclareRobustCommand\citepos
  {\begingroup
    \let\NAT@nmfmt\NAT@posfmt
    \NAT@swafalse\let\NAT@ctype\z@\NAT@partrue
    \@ifstar{\NAT@fulltrue\NAT@citetp}{\NAT@fullfalse\NAT@citetp}}
 
\let\NAT@orig@nmfmt\NAT@nmfmt
\def\NAT@posfmt#1{\NAT@orig@nmfmt{#1's}}

\makeatother

\newtheorem{lemma}{Lemma}
 
\newtheorem{theorem}{Theorem}[section]

\newtheorem{example}{Example}[section]

\newcommand{\ba}{\begin{align}} \newcommand{\ea}{\end{align}}
\newcommand{\baa}{\begin{align*}} \newcommand{\eaa}{\end{align*}}
\newcommand{\ben}{\begin{enumerate}} \newcommand{\een}{\end{enumerate}}
\newcommand{\bi}{\begin{itemize}} \newcommand{\ei}{\end{itemize}}

\newcommand{\ud}{\mathrm{d}}
\newcommand{\E}[1]{\operatorname{E}\left[ #1 \right]}

\newcommand{\var}[1]{\operatorname{Var}\left[ #1 \right]}

\newcommand{\cov}[2]{\operatorname{Cov}\left[ #1,#2 \right]}
\newcommand{\corr}[2]{\operatorname{Corr}\left[ #1,#2 \right]}

\newenvironment{proof}{\noindent\\ \noindent\relax{\sc
  Proof}}{{\samepage\par\nopagebreak\hbox
  to\hsize{\hfill$\Box$}}}
\newcommand{\be}{\begin{equation}} \newcommand{\ee}{\end{equation}}
\newcommand{\bd}{\begin{displaymath}} \newcommand{\ed}{\end{displaymath}}

\newcommand{\proglang}[1]{\texttt{#1}}

\title{First two moments and cross--moments of some coalescent times of the pure birth tree}
\author{Krzysztof 
 Bartoszek$^{1,\ast}$, Bayu Brahmantio$^{1,\dag}$, Woodrow Hao Chi Kiang$^{1,\ddag}$
  \\
  \noindent {\small \it 
  $^{1}$Department of Computer and Information Science,
  } \\
  {\small \it
   Link\"oping University, Link\"oping, Sweden}
  \\
  {\small \it  e--mail: $^{\ast}$krzysztof.bartoszek@liu.se; krzbar@protonmail.ch }
  \\  {\small \it $^{\dag}$ bayu.brahmantio@liu.se}
  \\ {\small  \it $^{\ddag}$  woodrow.hao.chi.kiang@liu.se; hello@hckiang.com}
}
\date{}

\begin{document}
\maketitle
\begin{abstract}
We present here a thorough study of the first two moments and cross--moments for the pure birth tree's height and
for the coalescent time of a randomly sampled pair of tips. We
consider also the first two moments of the conditional, on the tree, expectation of this coalescent time.
\end{abstract}

\renewcommand{\thesection}{Y\arabic{section}}
\setcounter{section}{0}
\section{Introduction}
We collect here a number of results concerning the pure--birth tree. They provide us with
information concerning the moments and cross--moments, in particular the first two,
of the tree's height and time to coalescent of a random pair.
Some of these results can be found in our previous works \citep{KBarSSag2015aart,KBarSSag2015bart,KBar2018art,SSagKBar2012art},
though many could be certainly traced further back.
Other formul\ae, to the best of our knowledge, are presented here for the first time.
We put all of them together in systematic manner that will allow for easy use in further applications. Results concerning the moments of these times are not only interesting in themselves but can be used in the study
of phylogenetic tree indices, e.g., the total cophenetic index
\citep[][]{KBar2018art,AMirFRosLRot2013art,SSagKBar2012art}.

We denote the generalized harmonic numbers as
$$H_{n,r}=\sum_{i=1}^{n}i^{-r}$$ 
recalling that $H_{n,1}\sim \ln n$, and $H_{n,2}\to \pi^{2}/6$.
Next, we denote for $x>-1$, and $n \in \mathbb{N}$, the following function,
$$
b_{n,x} = \frac{1}{x+1}\cdot\frac{2}{2+x}\cdot \ldots \cdot \frac{n}{n+x} = \frac{\Gamma(n+1)\Gamma(x+1)}{\Gamma(n+x+1)},
$$
where $\Gamma(\cdot)$ is the gamma function. 
In the derivations, we reference some lemmata involving harmonic sums with an ``H.'' prefix, 
these refer to lemmata that were presented by \cite{KBar2023arXivH},
who collected a number of closed form formul\ae\ for harmonic and quadratic
harmonic sums.

\section{Random variables associated with heights of the pure--birth tree}\label{secYuleRVs}
\renewcommand{\thelemma}{\thesection.\arabic{lemma}}
\renewcommand{\thetheorem}{\thesection.\arabic{theorem}}
\renewcommand{\theremark}{\thesection.\arabic{remark}}
\renewcommand{\theexample}{\thesection.\arabic{example}}
\renewcommand{\thecorollary}{\thesection.\arabic{corollary}}
\setcounter{equation}{0}
\setcounter{lemma}{0}
\setcounter{theorem}{0}
\setcounter{example}{0}
\setcounter{remark}{0}
\setcounter{corollary}{0}

We consider a pure birth branching process with speciation rate $\lambda=1$ (changing this 
value is equivalent to rescaling time, hence it will not have any qualitative effect on the 
results presented here). For the purpose of this work, we call such a tree a Yule tree with rate $\lambda=1$. The realization of such a branching process is a tree structure. 
We condition the tree to be stopped just before the $n$--th speciation event,
i.e., there are $n+1$ pendant branches in the tree (a single origin branch and $n$ leading to 
tips present at the stopping time). The key property of this process is that time intervals between
speciation times are independent and the time between the $k$--th and $(k+1)$--st speciation
event is exponentially distributed with rate $k$ (as the minimum of $k$ exponential, with rate
$1$, random variables).
We denote by $U^{(n)}$ the (random) height of the Yule tree, and by
and $\tau^{(n)}$
the time to coalescent (backwards from today) of a random pair of tip species (see Fig. \ref{fig:phyltreeRV}). 
We also introduce the $\sigma$--field $\mathcal{Y}_{n}$, that contains information on the tree,
i.e., conditional on $\mathcal{Y}_{n}$ the tree, its topology and branch lengths are known.
By the already mentioned property we have $U^{(n)} = T_{1}+\ldots+T_{n}$, where $T_{i} \sim \exp(i)$ are independent.
This immediately gives us that 
$$ \E{U^{(n)} } =\E{\sum\limits_{i=1}^{n}T_{i}} = H_{n,1}~\mathrm{and}~\E{\sum\limits_{i=1}^{n}T_{i}^{2}} = 2H_{n,2}.$$ 
Let $\kappa_{n}\in \{1,\ldots,n-1\}$ be the random variable representing the speciation event at which 
a randomly selected pair of tips of the Yule tree coalesced (see Fig. \ref{fig:phyltreeRV}). We know (e.g., \citepos{KBarSSag2015aart} Lemma 1, or \citet{TreeSim1}, \citet{TreeSim2})
$$
P(\kappa_{n}=k) = \frac{2(n+1)}{n-1}\frac{1}{(k+1)(k+2)} =: \pi_{n,k}
$$
and now we can write
$$
\tau^{(n)} = T_{\kappa_n+1} + \ldots + T_{n}.
$$
We further notice that $\E{U^{(n)}\vert \mathcal{Y}_{n}}= U^{(n)}$, as conditional on 
$\mathcal{Y}_{n}$ all times, in particular $T_{1},\ldots,T_{n}$, are known. The same does
not hold for $\E{\tau^{(n)}\vert \mathcal{Y}_{n}}$, as the choice of the random pair 
of tips does not belong to $\mathcal{Y}_{n}$.

We also need to consider which speciation events are on a randomly sampled lineage. Denote, after \citet{KBar2018art}, $1_{k}^{(n)}$ as the indicator random variable that a randomly sampled pair of tips from a pure--birth tree on $n$ leaves coalesced at the $k$--th (counting from the origin of the tree, see Fig. \ref{fig:phyltreeRV}) speciation event; we recall that 
$\E{1_{k}^{(n)}}=\pi_{k,n}$. When we write the pair 
$1_{k,1}^{(n)}$, $1_{k,2}^{(n)}$ we will mean two independent copies of $1_{k}^{(n)}$, i.e., we sample a a pair of tips twice---and ask if both pairs coalesced at the $k$--th speciation event.

\begin{figure}
    \centering
    \includegraphics[width=0.8\textwidth]{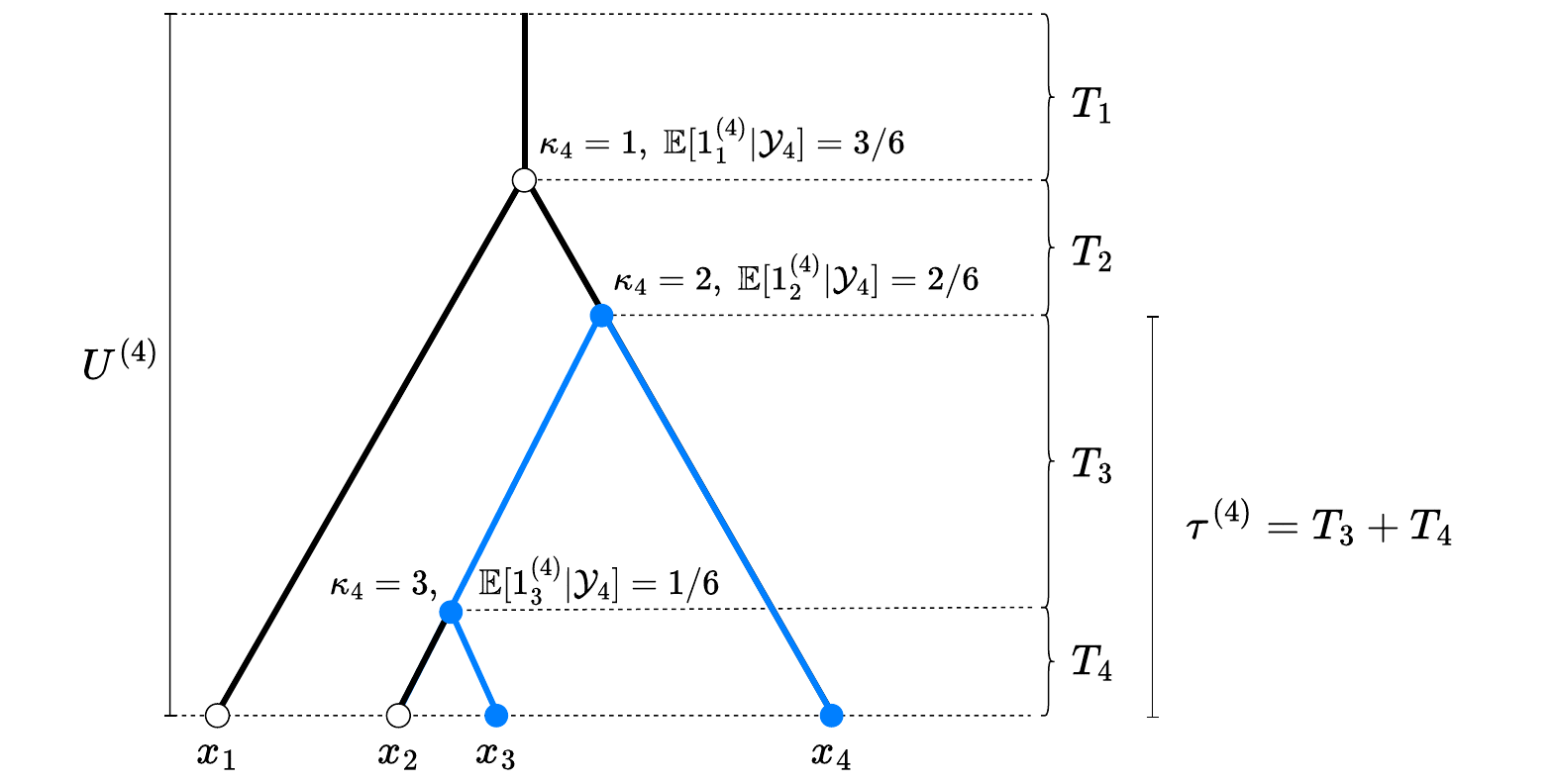}
    \caption{ A pure–birth tree, stopped just before the fifth speciation event, illustrating the various coalescent components we consider. We have four tips, $n=4$, and we ``randomly sample'' tips $x_{3}$ and $x_{4}$, and $\tau^{(4)}$ is their coalescent time. The interspeciation times, $T_{i}$s are distributed as $T_{i}\sim \exp(i)$ in our case. 
    The pairs $\{(x_{1},x_{2}),(x_{1},x_{3}),(x_{1},x_{4})\}$ 
    coalesced at $\kappa_{n}=1$; at $\kappa_{n}=2$
    the pairs $\{(x_{2},x_{4}),(x_{3},x_{4})\}$; and 
    at $\kappa_{n}=3$  only the pairs $\{(x_{2},x_{3})\}$.
}
    \label{fig:phyltreeRV}
\end{figure}

\section{Summary of results concerning moments of heights of the pure--birth tree}\label{secYuleMomentsSummary}
\renewcommand{\thelemma}{\thesection.\arabic{lemma}}
\renewcommand{\thetheorem}{\thesection.\arabic{theorem}}
\renewcommand{\theremark}{\thesection.\arabic{remark}}
\renewcommand{\theexample}{\thesection.\arabic{example}}
\renewcommand{\thecorollary}{\thesection.\arabic{corollary}}
\setcounter{equation}{0}
\setcounter{lemma}{0}
\setcounter{theorem}{0}
\setcounter{example}{0}
\setcounter{remark}{0}
\setcounter{corollary}{0}

In this section we collect a number of key first two moment properties of $U^{(n)}$, $\tau^{(n)}$, and
$\E{\tau^{(n)}\vert \mathcal{Y}_{n}}$. Some of them can be found in the literature, and some
we prove here. 
These, and related formul\ae\ are needed when proving properties of processes evolving on a pure--birth tree, central limit theorems
\citep{KBarSSag2015aart, KBar2020art,KBarTErh2021art}, the sum of branch lengths \citep{KBar2014art}, or the so--called cophenetic index 
\citep{SSagKBar2012art,KBar2018art}.
\begin{enumerate}
\item $\E{U^{(n)}}=H_{n,1}$ \citep[e.g.,][but this is essentially common knowledge]{SSagKBar2012art};
\item $U^{(n)}-\log n$ has a Gumbel limiting distribution \citep{KBarSSag2015aart};
\item $\E{U^{(n)^{2}}}=H_{n,1}^{2}+H_{n,2}$ (setting $k=0$ after Eq. $(8)$ on  p. $74$ in \citep{KBarSSag2015bart}, or Ex. \ref{exEUnm1m2}, but this is essentially common knowledge);
\item $\E{\tau^{(n)}}= \frac{n+1}{n-1}H_{n,1}-\frac{2n}{n-1}$ \citep{SSagKBar2012art};
\item $\E{U^{(n)}-\tau^{(n)}}=2(n-H_{n,1})/(n-1)$ \citep{SSagKBar2012art};
\item $\E{U^{(n)}(U^{(n)}-\tau^{(n)})}= \frac{2n}{n-1}\left(H_{n,1}+H_{n,2} -1-n^{-1}H_{n,1}^{2}\right)$ (Lemma \ref{lemUnUnmTaun});
\item $\E{U^{(n)}(U^{(n)}-\E{\tau^{(n)}\vert \mathcal{Y}_{n}})}=\frac{2n}{n-1}\left(H_{n,1}+H_{n,2} -1-n^{-1}H_{n,1}^{2}\right)$ 
(Lemma \ref{lemUnUnmETaunYn});
\item $\E{\tau^{(n)^{2}}}= \frac{n+1}{n-1}\left((H_{n,1}-2)^{2}+H_{n,2} \right)-\frac{4}{n-1}$ (Ex. \ref{exEtaunm1m2});
\item $\E{(U^{(n)}-\tau^{(n)})^{2}}= \frac{4n+2}{n-1}H_{n,2} -\frac{2H_{n,1}^{2}}{n-1}-\frac{4H_{n,1}}{n-1}$
(Appendix A in \citep{KBarSSag2015aart},  p. $35$ in \citep{KBar2018art} and Lemma \ref{lemUnmTaun2} for the final simplified form);
\item $\E{\E{\tau^{(n)}\vert \mathcal{Y}_{n}}^{2}} = 
H_{n,1}^{2}-4H_{n,1}- \frac{5H_{n,2}}{3} +7
+\frac{4H_{n,1}^{2}}{n}
+\Theta(n^{-1}\ln n)$ \\ (Lemma \ref{lemETaunYn2}); 
\item 
$
\E{(U^{(n)}-\E{\tau^{(n)}\vert \mathcal{Y}_{n}})^{2}} = 
\frac{4}{3}H_{n,2}+3-8n^{-1}H_{n,1}+\Theta(n^{-1})
\xrightarrow{n\to \infty} \frac{2\pi^{2}}{9}+3 ,
$ 
as $n\to\infty$ (Lemma \ref{lemUnmETaunYn2}); 
\item $\E{(\tau^{(n)}-\E{\tau^{(n)}\vert \mathcal{Y}_{n}})^{2}}
=\E{((U^{(n)}-\tau^{(n)})-(U^{(n)}-\E{\tau^{(n)}\vert \mathcal{Y}_{n}}))^{2}}
\\= 
\frac{8}{3}H_{n,2}-3-\frac{2H_{n,1}^{2}}{(n-1)}
+ \Theta(n^{-1}\ln n)
\xrightarrow{n\to \infty}
\frac{4\pi^{2}}{9} - 3
$
(Lemma \ref{lemtaunmETaunYn2}). 
\end{enumerate}

We can see that the shared path length, $U^{(n)}-\tau^{(n)}$, for a random pair of tip species, has mean value 
converging to $2$ and variance to $4(\pi^{2}/6-1)$.
This means that we expect that on average 
the shared path length between a pair of species is small compared the tree's height, which has expectation behaving as $\ln n$ and variance converging to
$\pi^{2}/6$.

As already mentioned, these heights are related to important tree statistics. For example the total cophenetic index (generalized to trees with branch lengths) will equal 
$\binom{n}{2}\left(U^{(n)}-\E{\tau^{(n)}\vert \mathcal{Y}_{n}}\right)$,
and in fact this was discussed and exploited by \citet{KBar2018art}.
Hence, e.g., from Lemma \ref{lemUnmETaunYn2} we can obtain the second moment, or combining Lemma \ref{lemUnmETaunYn2}
with $\E{U^{(n)}-\tau^{(n)}}$ easily gives the variance. 

\section{Moments of heights of the pure--birth tree}\label{secYuleMoments}
\renewcommand{\theequation}{\thesection.\arabic{equation}}
\renewcommand{\thelemma}{\thesection.\arabic{lemma}}
\renewcommand{\thetheorem}{\thesection.\arabic{theorem}}
\renewcommand{\theremark}{\thesection.\arabic{remark}}
\renewcommand{\theexample}{\thesection.\arabic{example}}
\renewcommand{\thecorollary}{\thesection.\arabic{corollary}}
\setcounter{equation}{0}
\setcounter{lemma}{0}
\setcounter{theorem}{0}
\setcounter{example}{0}
\setcounter{remark}{0}
\setcounter{corollary}{0}

\begin{theorem}[Appendix A in \cite{KBarSSag2015aart}]\label{thmMomkUn}
Define the set of integer valued vectors 
$$
\mathcal{A}_{m} = \{ \mathbf{k} : \sum\limits_{i=1}^{m}k_{i}i=m \},
$$
i.e., $\mathcal{A}_{m}$ is the set of all possible ways to represent $m$ as a sum of positive integers.
The $m$--th moment of the tree height of a Yule tree with speciation rate $1$ is
\be\label{eqEUnm}
\E{\left(U^{(n)}\right)^{m}} = \sum\limits_{\mathbf{k} \in \mathcal{A}_{m}} c_{\mathbf{k}}
\prod\limits_{i=1}^{m}H_{n,i}^{k_{i}},
\ee
where for a non--negative integer valued vector $\mathbf{k}=(k_{1},k_{2},\ldots)$, such that $m=\sum_{i=1}^{m}k_{i}i$ 
we define $c_{\mathbf{k}}$ recursively as
\be\label{eqck}
c_{\mathbf{k}} = c_{\mathbf{k},0}+\sum\limits_{j=1}^{m-1}(j(k_{j}+1))c_{\mathbf{k},j},
\ee
and
$$
\begin{array}{c}
c_{\mathbf{k},0}=c_{(k_{1}-1,k_{2},k_{3},\ldots)} \\
c_{\mathbf{k},j}=c_{(k_{1},\ldots,k_{j}+1,k_{j+1}-1,\ldots)},~j\ge 1.
\end{array}
$$
The boundary conditions for Eq. \eqref{eqck} are
\begin{enumerate}
\item $c_{\mathbf{k}}=0$ if all $k_{i}=0$ or one of the coordinates of $\mathbf{k}$ is negative,
\item $c_{\mathbf{k}}=1$ if $k_{1}\ge 1$ and all other $k_{i}=0$.
\end{enumerate}
\end{theorem}
\begin{proof}
The proof of this theorem can be found in 
Appendix A in \cite{KBarSSag2015aart}. For the convenience of the reader,
especially as its steps are helpful in proving the next Thm. \ref{thmMomktaun}, 
we provide here a full detailed proof of it
(let us mention that in \citepos{KBarSSag2015aart} Eq. (11), ours Eq. \eqref{eqck} a bracket around $k_{j}+1$ is not properly set). 
One can find in \citepos{KBarSSag2015aart} Lemma $3$ that the Laplace transform of the height of the Yule tree is
$$
\E{e^{-x U^{(n)}}} = b_{n,x}.
$$
We will use the property that for a random variable $Z$, with Laplace transform $\mathcal{L}_{Z}(x)$, its $m$--th moment is obtained
as
$$
\E{Z^{m}} = (-1)^{m} \frac{\ud^{m} \mathcal{L}_{Z}}{\ud x^{m}}\vert_{x=0}.
$$
We introduce the notation (for a fixed $n$)
$$
A(x) = b_{n,x}/n! = \frac{1}{x+1}\cdot \ldots \cdot \frac{1}{x+n},~~b_{m}(x) = \frac{1}{(x+1)^{m}}+ \ldots + \frac{1}{(x+n)^{m}}
$$
and also the vector $\mathbf{b}_{m}(x) = (b_{1}(x),\ldots,b_{m}(x))$. We 
can immediately notice that $A(0)=1/n!$ and $b_{m}(0)=H_{n,m}$. The first and second derivatives of $A(x)$ are
$A'(x) = -A(x)b_{1}(x)$ and $A''(x) = A(x)(b_{1}^{2}(x)+b_{2}(x))$, notice that $b'_{m}(x)=-mb_{m+1}(x)$.
We also introduce the multi--index notation
$\mathbf{b}_{m}(x)^{\mathbf{k}} = b_{1}(x)^{k_{1}}\cdot\ldots\cdot b_{m}(x)^{k_{m}}$. 
One then shows inductively
that 
\be\label{eqAderiv}
\frac{\ud^{m}}{\ud x^{m}}A(x) = 
(-1)^{m}A(x)\sum\limits_{\mathbf{k} \in \mathcal{A}_{m}} c_{\mathbf{k}}\mathbf{b}_{m}(x)^{\mathbf{k}}.
\ee
We can see directly that Eq. \eqref{eqAderiv} holds for $m=1$. Assume that it holds till $m$ and the $m+1$st derivative is
$$
\begin{array}{l}
\frac{\ud^{m+1}}{\ud x^{m}}A(x) = \left(\frac{\ud^{m}}{\ud x^{m}}A(x) \right)'
= 
\left( (-1)^{m}A(x)\sum\limits_{\mathbf{k} \in \mathcal{A}_{m}} c_{\mathbf{k}}\mathbf{b}_{m}(x)^{\mathbf{k}}\right)'
\\=
(-1)^{m}
\left(A'(x)\sum\limits_{\mathbf{k} \in \mathcal{A}_{m}} c_{\mathbf{k}}\mathbf{b}_{m}(x)^{\mathbf{k}}
+
A(x)\sum\limits_{\mathbf{k} \in \mathcal{A}_{m}} c_{\mathbf{k}}(\mathbf{b}_{m}(x)^{\mathbf{k}})'\right)
\\=
(-1)^{m+1}A(x)
\left(
\sum\limits_{\mathbf{k} \in \mathcal{A}_{m}} c_{\mathbf{k}}
\left(\mathbf{b}_{m}(x)^{(k_{1}+1,k_{2},\ldots)} +
\sum\limits_{i=1}^{m}  ik_{i}\mathbf{b}_{m+1}^{(k_{1},\ldots,k_{i}-1,k_{i+1}+1,\ldots)} \right)\right).
\end{array}
$$
As, 
$k_{1}+i(k_{i}-1)+(i+1)(k_{i+1}+1)+\ldots=k_{1}+ik_{i}+(i+1)k_{i+1}+\ldots+i+1-i=k_{1}+ik_{i}+(i+1)k_{i+1}+\ldots+1=m+1$,
for any given $\mathbf{k} \in \mathcal{A}_{m+1}$ 
we will obtain
in front of $\mathbf{b}_{m+1}(x)^{\mathbf{k}}$
the coefficient 
$$
c_{(k_{1}-1,k_{2},\ldots)} + \sum\limits_{i=1}^{m} c_{(k_{1},\ldots,k_{i}+1,k_{i+1}-1,\ldots)}i(k_{i}+1)
\equiv c_{\mathbf{k},0}+\sum\limits_{i=1}^{m} c_{\mathbf{k},i}i(k_{i}+1),
$$
as desired. 
Evaluating the derivative at $0$ we obtain
$$
\frac{\ud^{m}}{\ud x^{m}}A(x)\vert_{x=0} = 
(-1)^{m}A(0)\sum\limits_{\mathbf{k} \in \mathcal{A}_{m}} c_{\mathbf{k}}\mathbf{b}_{m}(0)^{\mathbf{k}}
=
(-1)^{m}\frac{1}{n!}\sum\limits_{\mathbf{k} \in \mathcal{A}_{m}} c_{\mathbf{k}}
\prod\limits_{i=1}^{m}H_{n,i}^{k_{i}}.
$$
To obtain the final statement recall that $\E{e^{-x U^{(n)}}}=b_{n,x}=n!A(x)$.
\end{proof}

\begin{example}\label{exEUnm1m2}
We recover here the first two moments, $m=1, 2$, for $U^{(n)}$. We first need to find $c_{\mathbf{k}}$ for 
$\mathbf{k} \in \mathcal{A}_{1}, \mathcal{A}_{2}$. The only $\mathbf{k}$ vectors
that need to be considered are $\mathbf{v}_{1}=(1,0,0,\ldots)$, $\mathbf{v}_{2}=(2,0,0,\ldots)$, $\mathbf{v}_{3}=(0,1,0,0\ldots)$.
We have directly from the boundary conditions that $c_{\mathbf{v}_{1}}=1$. We calculate for $\mathbf{v}_{2}$
$c_{\mathbf{v}_{2}}= c_{\mathbf{v}_{2},0}+1(2+1)c_{\mathbf{v}_{2},1}$,
$$
\begin{array}{l}
c_{\mathbf{v}_{2},0} = c_{(1,0,\ldots)}=1, \\
c_{\mathbf{v}_{2},1} = c_{(3,-1,0,\ldots)}=0,
\end{array}
$$
resulting in $c_{\mathbf{v}_{2}}=1$. For $\mathbf{v}_{3}$
$c_{\mathbf{v}_{3}}= c_{\mathbf{v}_{3},0}+1(0+1)c_{\mathbf{v}_{2},1}$,
$$
\begin{array}{l}
c_{\mathbf{v}_{3},0} = c_{(-1,1,\ldots)}=0, \\
c_{\mathbf{v}_{3},1} = c_{(1,0,0,\ldots)}=1,
\end{array}
$$
resulting in $c_{\mathbf{v}_{3}}=1$. From Eq. \eqref{eqEUnm} we have
$$
\begin{array}{l}
\E{U^{(n)}} = \sum\limits_{\mathbf{k} \in \mathcal{A}_{1}} c_{\mathbf{k}} \prod\limits_{i=1}^{1}H_{n,i}^{k_{i}}
= c_{\mathbf{v}_{1}} H_{n,1}^{1} =  H_{n,1}, \\
\E{U^{(n)^{2}}} = \sum\limits_{\mathbf{k} \in \mathcal{A}_{2}} c_{\mathbf{k}} \prod\limits_{i=1}^{2}H_{n,i}^{k_{i}}
= c_{\mathbf{v}_{2}} H_{n,1}^{2}H_{n,2}^{0} + c_{\mathbf{v}_{3}} H_{n,1}^{0}H_{n,2}^{1}  = H_{n,1}^{2}+H_{n,2}.
\end{array}
$$
\end{example}

\begin{theorem}[Appendix A \cite{KBarSSag2015aart}]\label{thmMomktaun}
The $m$--th moment of the time to coalescent of a random pair of tips in a Yule tree with speciation rate $1$,
with the same notation as in Thm. \ref{thmMomkUn}, is
\be\label{eqEtaunm}
\E{\left(\tau^{(n)}\right)^{m}} = 
(-1)^{m+1}\frac{2m!}{(n-1)}+\frac{n+1}{n-1}
\sum\limits_{\mathbf{k} \in \mathcal{A}_{m}} c_{\mathbf{k}}
\left(\prod\limits_{\begin{array}{c} i=1\\ i \mathrm{~even}\end{array}}^{m}H_{n,i}^{k_{i}}\right)
\left(\prod\limits_{\begin{array}{c} i=1\\ i \mathrm{~odd}\end{array}}^{m}(H_{n,i}-2)^{k_{i}}\right).
\ee
\end{theorem}
\begin{proof}
For the convenience of the reader we provide a full detailed proof here (\citet{KBarSSag2015aart}
considered the joint moment $\E{(U^{(n)}-\tau^{(n)})^{m}\left(\tau^{(n)}\right)^{r}}$, and 
we also notify that in  \citepos{KBarSSag2015aart} Eq. \eqref{eqEtaunm}, e.g.,  the factor $(-1)^{m+1}$ is missing, what may confuse a bit).
The proof will come along the lines of the proof of Thm. \ref{thmMomkUn}. We first recall that the Laplace transform
for $\tau^{(n)}$ is (\citepos{KBarSSag2015aart} Lemma $3$)
\be\label{eqLaptaun}
\E{e^{-x\tau^{(n)}}} =  \frac{2-(n+1)(x+1)b_{n,x}}{(n-1)(x-1)}= \frac{2}{(n-1)}\frac{1}{(x-1)}-\left(\frac{n+1}{n-1}\right)\frac{(x+1)}{(x-1)}b_{n,x}.
\ee
We are interested in the behaviour around $0$ so the singularity at $x=1$ is not relevant here. 
The $m$--th derivative of $(x-1)^{-1}$ is 
$$
\left(\frac{1}{x-1}\right)^{(m)} = (-1)^{m}\frac{m!}{(x-1)^{m+1}}.
$$
One needs to now consider the $m$--th derivative of $(x+1)b_{n,x}/(x-1)$. Define now $\hat{A}(x)=(x+1)A(x)/(x-1)$ and 
$$\hat{b}_{1}(x)=b_{1}(x)+2\frac{1}{(x+1)(x-1)}.$$
The $m$--th derivative of $1/((x-1)(x+1))$ is 
\be\label{eqderivpart2}
\left(\frac{1}{(x-1)(x+1)}\right)^{(m)} = (-1)^{m}m!\sum\limits_{i=1}^{m+1}\frac{1}{(x+1)^{i}(x-1)^{m+2-i}}.
\ee
Defining generally,
$$
\begin{array}{l}
\hat{b}_{m}(x)=b_{m}(x)+2\sum\limits_{i=1}^{m}\frac{1}{(x+1)^{i}(x-1)^{m+1-i}}
=
b_{m}(x)+\left(\frac{1}{(x-1)^{m}}-\frac{1}{(x+1)^{m}}\right).
\end{array}
$$
The second component,
$$
\sum\limits_{i=1}^{m}\frac{1}{(x+1)^{i}(x-1)^{m+1-i}}=\frac{1}{2}\left(\frac{1}{(x-1)^{m}}-\frac{1}{(x+1)^{m}}\right),
$$
can be checked by induction on $m$. 
Hence, we can see that $\hat{b}_{m+1}'(x)=-m\hat{b}_{m+1}(x)$.
Now we turn to calculating the $m$--th derivative of $(\hat{A}(x))^{(m)}$.
We have 
$$
\begin{array}{l}
\hat{A}'(x)=
\left(\frac{x+1}{x-1} \right)'A(x) + \frac{x+1}{x-1}A'(x)
=
(-1)\left(2\frac{1}{(x-1)(x+1)} + b_{1}(x)\right)\hat{A}(x)
=
(-1)\hat{A}(x)\hat{b}_{1}(x).
\end{array}
$$
Hence, as $\hat{A}'(x) = (-1)\hat{b}_{1}(x)\hat{A}(x)$ and $\hat{b}_{m+1}'(x)=-m\hat{b}_{m+1}(x)$
we have the same situation as in Thm. \ref{thmMomkUn}. Therefore, we may directly write
\be\label{eqAhatderiv}
\frac{\ud^{m}}{\ud x^{m}}\hat{A}(x) = 
(-1)^{m}\hat{A}(x)\sum\limits_{\mathbf{k} \in \mathcal{A}_{m}} c_{\mathbf{k}}\mathbf{\hat{b}}_{m}(x)^{\mathbf{k}}.
\ee
Now, we obtain that the
$m$--th derivative of the Laplace transform of $\tau^{(n)}$ will be 
$$
\begin{array}{l}
\left(\E{e^{-x\tau^{(n)}}}\right)^{(m)} =  
\frac{2}{(n-1)}\left(\frac{1}{(x-1)}\right)^{(m)}-\left(\frac{n+1}{n-1}\right)\left(\frac{(x+1)}{(x-1)}b_{n,x}\right)^{(m)}
\\=
\frac{2}{(n-1)}\left(\frac{1}{(x-1)}\right)^{(m)}-\left(\frac{n+1}{n-1}\right)n!\hat{A}(x)^{(m)}
\\ =
(-1)^{m}\left(
\frac{2}{(n-1)}\frac{m!}{(x-1)^{m+1}} - \frac{(n+1)!}{n-1}\hat{A}(x)\sum\limits_{\mathbf{k} \in \mathcal{A}_{m}} c_{\mathbf{k}}\mathbf{\hat{b}}_{m}(x)^{\mathbf{k}}
\right).
\end{array}
$$
To evaluate the derivative at $0$ we notice that 
$\hat{A}(0)=-A(0)=1/n!$ and 
$$
\begin{array}{l}
\hat{b}_{m}(0)=b_{m}(0)+((-1)^{m}-1)
=H_{n,m}+((-1)^{m}-1)
\\ =\left\{
\begin{array}{lc}
H_{n,m} & m\mathrm{~is~even}, \\
H_{n,m}-2 & m\mathrm{~is~odd}. \\
\end{array}\right.
\end{array}
$$
Therefore, 
$$
\E{\tau^{(n)^{m}}} = (-1)^{m+1}\frac{2m!}{(n-1)}+\frac{n+1}{n-1}
\sum\limits_{\mathbf{k} \in \mathcal{A}_{m}} c_{\mathbf{k}}
\left(\prod\limits_{\begin{array}{c} i=1\\ i \mathrm{~even}\end{array}}^{m}H_{n,i}^{k_{i}}\right)
\left(\prod\limits_{\begin{array}{c} i=1\\ i \mathrm{~odd}\end{array}}^{m}(H_{n,i}-2)^{k_{i}}\right).
$$
\end{proof}
\begin{example}\label{exEtaunm1m2}
We recover here the first two moments $m=1, 2$, for $\tau^{(n)}$. We have the same $c_{\mathbf{v}_{1}}$, 
$c_{\mathbf{v}_{2}}$ and $c_{\mathbf{v}_{3}}$ as in Example \ref{exEUnm1m2}.
From Eq. \eqref{eqEtaunm} we have 
$$
\begin{array}{l}
\E{\tau^{(n)}} = 
(-1)^{1+1}\frac{2}{(n-1)}1!+\frac{n+1}{n-1}
\sum\limits_{\mathbf{k} \in \mathcal{A}_{1}} c_{\mathbf{k}}
\left(\prod\limits_{\begin{array}{c} i=1\\ i \mathrm{~even}\end{array}}^{1}H_{n,i}^{k_{i}}\right)
\left(\prod\limits_{\begin{array}{c} i=1\\ i \mathrm{~odd}\end{array}}^{1}(H_{n,i}-2)^{k_{i}}\right) 
\\=
\frac{2}{n-1}+\frac{n+1}{n-1} c_{\mathbf{v}_{1}}(H_{n,i}-2)^{1}
=\frac{2}{n-1}+\frac{n+1}{n-1} (H_{n,i}-2),
\\
\E{\tau^{(n)^{2}}} = 
(-1)^{2+1}\frac{2}{(n-1)}2!+\frac{n+1}{n-1}
\sum\limits_{\mathbf{k} \in \mathcal{A}_{2}} c_{\mathbf{k}}
\left(\prod\limits_{\begin{array}{c} i=1\\ i \mathrm{~even}\end{array}}^{2}H_{n,i}^{k_{i}}\right)
\left(\prod\limits_{\begin{array}{c} i=1\\ i \mathrm{~odd}\end{array}}^{2}(H_{n,i}-2)^{k_{i}}\right)
\\ =
\frac{(-4)}{(n-1)}+\frac{n+1}{n-1}\left(
c_{\mathbf{v}_{2}}H_{n,2}^{0}(H_{n,1}-2)^{2}
+c_{\mathbf{v}_{3}}H_{n,2}^{1}(H_{n,1}-2)^{0}
\right)
\\ =
\frac{n+1}{n-1}\left((H_{n,1}-2)^{2}+H_{n,2}\right)-\frac{4}{(n-1)}.
\end{array}
$$
\end{example}

\begin{lemma}[See also the proof of Thm. $5.1$ in \citep{KBar2018art}]\label{lemUnmTaun2}
For a Yule tree with speciation rate $\lambda=1$ we have
\be\label{eqUnmTaun2}
\E{(U^{(n)}-\tau^{(n)})^{2}}= \frac{4n+2}{n-1}H_{n,2} -\frac{2H_{n,1}^{2}}{n-1}-\frac{4H_{n,1}}{n-1}
\ee
\end{lemma}
\begin{proof}
We will derive the second moment 
from the general joint moments formula,
with an alternative proof in \ref{appAltProof}. 
Using the same notation as in Thm. \ref{thmMomkUn}, we first recall from 
\citepos{KBarSSag2015aart} Appendix A
the general joint moment formula (setting $r=0$)
\be\label{eqEUnmtaunmGen}
\E{(U^{(n)}-\tau^{(n)})^{m}} = (-1)^{m}\frac{2(n+1)}{(n-1)}\sum\limits_{j=1}^{n-1}\frac{1}{(j+1)(j+2)}\left(\sum\limits_{\mathbf{k} \in \mathcal{A}_{m}}c_{\mathbf{k}}\prod\limits_{i=1}^{m}H_{j,i}^{k_{i}} \right).
\ee
Setting $m=2$, recalling Example \ref{exEUnm1m2}, and using 
Lemmata \ref{HarXiv-lemSumHi2i1i2} and \ref{HarXiv-lemSumHisqi1i2app} we may write

$$
\begin{array}{l}
\E{(U^{(n)}-\tau^{(n)})^{2}} = \frac{2(n+1)}{(n-1)}\sum\limits_{j=1}^{n-1}\frac{1}{(j+1)(j+2)}\left(\sum\limits_{\mathbf{k} \in \mathcal{A}_{2}}c_{\mathbf{k}}\prod\limits_{i=1}^{2}H_{j,i}^{k_{i}} \right)
\\ =\frac{2(n+1)}{(n-1)}\sum\limits_{j=1}^{n-1}\frac{1}{(j+1)(j+2)}
\left(c_{\mathbf{v}_{2}} H_{j,1}^{2}H_{j,2}^{0} + c_{\mathbf{v}_{3}} H_{j,1}^{0}H_{j,2}^{1}\right)
 =\frac{2(n+1)}{(n-1)}
\left(\sum\limits_{j=1}^{n-1}\frac{H_{j,1}^{2}}{(j+1)(j+2)}+\sum\limits_{j=1}^{n-1}\frac{H_{j,2}}{(j+1)(j+2)}
\right)
\\
\stackrel{\mathrm{Lemmata~\ref{HarXiv-lemSumHi2i1i2}, \ref{HarXiv-lemSumHisqi1i2app}}}{=}\frac{2(n+1)}{(n-1)}
\left(
H_{n,2}
+\frac{(n-1)^{2}+n-1-1}{n^{2}}
- 2\frac{H_{n-1,1}}{n} -\frac{H_{n-1,1}^{2}}{n+1} 
+ \frac{n}{n+1}\left(H_{n,2}-1\right)
\right)
\\ =
\frac{4n+2}{n-1}H_{n,2}-\frac{4H_{n,1}}{n-1}-\frac{2H_{n,1}^{2}}{n-1}.
\end{array}
$$
This equation is compared with simulations in Fig. \ref{fig:fig1Lem1_6}.
\end{proof}

\begin{lemma}\label{lemUnTaun}
For a Yule tree with speciation rate $\lambda=1$ we have 
\be\label{eqUnTaun}
\E{U^{(n)}\tau^{(n)}}= \frac{n+1}{n-1}\left(H_{n,1}^{2}-H_{n,2} \right)
-\frac{2n}{n-1}\left(H_{n,1} -1\right).
\ee
\end{lemma}
\begin{proof}
We can calculate the cross--moment using Examples \ref{exEUnm1m2}, \ref{exEtaunm1m2}, and Lemma \ref{lemUnmTaun2}. 
We have

$$
\begin{array}{l}
\E{U^{(n)}\tau^{(n)}}= \frac{1}{2}\left(\E{U^{(n)^{2}}}+\E{\tau^{(n)^{2}}}-\E{(U^{(n)}-\tau^{(n)})^{2}} \right)
\\=
\frac{1}{2}\left(
H_{n,1}^{2}+H_{n,2} + \frac{n+1}{n-1}\left((H_{n,1}-2)^{2}+H_{n,2}\right)-\frac{4}{(n-1)}
-\frac{4n+2}{n-1}H_{n,2}+\frac{4H_{n,1}}{n-1}+\frac{2H_{n,1}^{2}}{n-1}
\right)
\\=
\frac{n+1}{n-1}\left(H_{n,1}^{2}-H_{n,2} \right)
-\frac{2n}{n-1}\left(H_{n,1} -1\right).
\end{array}
$$
This equation is compared with simulations in Fig. \ref{fig:fig1Lem1_6}.
\end{proof}

\begin{lemma}\label{lemUnUnmTaun}
For a Yule tree with speciation rate $\lambda=1$ we have 
\be\label{eqUnUnmTaun}
\E{U^{(n)}(U^{(n)}-\tau^{(n)})}= \frac{2n}{n-1}\left(H_{n,1}+H_{n,2} -1-n^{-1}H_{n,1}^{2}\right).
\ee
\end{lemma}
\begin{proof}
Using Lemma \ref{lemUnTaun} and Example \ref{exEUnm1m2}

$$
\begin{array}{l}
\E{U^{(n)}(U^{(n)}-\tau^{(n)})}= \E{U^{(n)^{2}}}-\E{U^{(n)}\tau^{(n)}}
\\= 
 H_{n,1}^{2}+H_{n,2}
-\frac{n+1}{n-1}\left(H_{n,1}^{2}-H_{n,2} \right)
+\frac{2n}{n-1}\left(H_{n,1} -1\right)
\\ =
\frac{2n}{n-1}\left(H_{n,1}+H_{n,2} -1-n^{-1}H_{n,1}^{2}\right).
\end{array}
$$
This equation is compared with simulations in Fig. \ref{fig:fig1Lem1_6}.
\end{proof}

\begin{lemma}\label{lemUnUnmETaunYn}
For a Yule tree with speciation rate $\lambda=1$ we have  
\be\label{eqUnUnmETaunYn}
\E{U^{(n)}(U^{(n)}-\E{\tau^{(n)}\vert \mathcal{Y}_{n}})}= \frac{2n}{n-1}\left(H_{n,1}+H_{n,2} -1-n^{-1}H_{n,1}^{2}\right).
\ee
\end{lemma}
\begin{proof}
$$
\begin{array}{l}
\E{U^{(n)}(U^{(n)}-\E{\tau^{(n)}\vert \mathcal{Y}_{n}})}= 
\E{\E{U^{(n)}(U^{(n)}-\tau^{(n)}\vert \mathcal{Y}_{n})}}
=\E{U^{(n)}(U^{(n)}-\tau^{(n)})}.
\end{array}
$$
This equation is compared with simulations in Fig. \ref{fig:fig1Lem1_6}.
\end{proof}

\begin{lemma}\label{lemUnmETaunYn2}
For a Yule tree with speciation rate $\lambda=1$ we have  
\be\label{eqUnmETaunYn2}
\begin{array}{l}
\E{(U^{(n)}-\E{\tau^{(n)}\vert \mathcal{Y}_{n}})^{2}} = 
\frac{4}{3}H_{n,2}+3-8n^{-1}H_{n,1}+\Theta(n^{-1})
\xrightarrow{n\to \infty} \frac{2\pi^{2}}{9}+3. 
\end{array}
\ee
\end{lemma}

\begin{proof}
We have from the second formula in \citepos{KBar2018art} Thm. $6.4$ 
$$
\begin{array}{l}
\var{U^{(n)}-\E{\tau^{(n)} \vert \mathcal{Y}_{n}}} 
= 
\frac{12n^{2}(n^{2}-6n-4)H_{n-1,2}-9n^{4}+102n^{3}+51n^{2}-24nH_{n-1,1}-72n-72}{9n^{2}(n-1)^{2}} 
\end{array}
$$
and from this we obtain \citep[recalling that $\E{U^{(n)}-\tau^{(n)}}=2(n-H_{n,1})/(n-1)$, ][]{SSagKBar2012art}

$$
\begin{array}{l}
\E{\left(U^{(n)}-\E{\tau^{(n)} \vert \mathcal{Y}_{n}}\right)^{2}} 
=\var{U^{(n)}-\E{\tau^{(n)} \vert \mathcal{Y}_{n}}} + \left(\E{U^{(n)}-\E{\tau^{(n)} \vert \mathcal{Y}_{n}}}\right)^{2}
\\ = \var{U^{(n)}-\E{\tau^{(n)} \vert \mathcal{Y}_{n}}} + \left(\E{U^{(n)}-\tau^{(n)} }\right)^{2}
=
\var{U^{(n)}-\E{\tau^{(n)} \vert \mathcal{Y}_{n}}} + \left(\frac{2(n-H_{n,1})}{n-1}\right)^{2}
\\ 
=
\frac{12n^{4}H_{n-1,2}-72n^{3}H_{n-1,2}-48n^{2}H_{n-1,2}-9n^{4}+102n^{3}+51n^{2}-24nH_{n-1,1}-72n-72}{9n^{2}(n-1)^{2}} 
+ \frac{36n^{4}-72n^{3}H_{n,1}+36n^{2}H_{n,1}^2}{9n^{2}(n-1)^{2}}
\\=
\frac{12n^{4}H_{n,2}-12n^{2}-72n^{3}H_{n,2}+72n-48n^{2}H_{n,2}+48
-9n^{4}+102n^{3}+51n^{2}-24nH_{n-1,1}-72n-72}{9n^{2}(n-1)^{2}} 
+ \frac{36n^{4}-72n^{3}H_{n,1}+36n^{2}H_{n,1}^2}{9n^{2}(n-1)^{2}}
\\=
\frac{12n^{4}H_{n,2}+27n^{4}
-72n^{3}H_{n,1}-72n^{3}H_{n,2}+102n^{3}
+36n^{2}H_{n,1}^{2}-48n^{2}H_{n,2}+39n^{2}
-24nH_{n,1}}{9n^{2}(n-1)^{2}}.
\end{array}
$$
We can simplify the above and obtain the asymptotics
\be\label{eqEUnmtaun2}
\begin{array}{l}
\E{\left(U^{(n)}-\E{\tau^{(n)} \vert \mathcal{Y}_{n}}\right)^{2}} =
\frac{4n^{3}H_{n,2}+9n^{3}
-24n^{2}H_{n,1}-24n^{2}H_{n,2}+34n^{2}
+12nH_{n,1}^{2}-16nH_{n,2}+13n
-8H_{n,1}}{3n(n-1)^{2}} 
\\ \xlongrightarrow{n\to \infty} \frac{2\pi^{2}}{9}+3 \approx 5.19325 .
\end{array}
\ee
This key equation is compared with simulations in Fig. \ref{fig:fig1Lem1_6}.
In \ref{appAltProof} we show an alternative, direct proof of this lemma.
\end{proof}

\begin{lemma}\label{lemETaunYn2}
For a Yule tree with speciation rate $\lambda=1$ we have  
\be\label{eqETaunYn2}
\E{\E{\tau^{(n)}\vert \mathcal{Y}_{n}}^{2}} = 
H_{n,1}^{2}-4H_{n,1}- \frac{5H_{n,2}}{3} +7
+\frac{4H_{n,1}^{2}}{n}
+\Theta(n^{-1}\ln n).
\ee
\end{lemma}
\begin{proof}
We may write
$$
\begin{array}{l}
\E{\E{\tau^{(n)}\vert \mathcal{Y}_{n}}^{2}} = 
\E{\left(U^{(n)}-U^{(n)}+\E{\tau^{(n)}\vert \mathcal{Y}_{n}}\right)^{2}} 
\\ =
\E{U^{(n)^{2}}}
-2\E{U^{(n)}\left(U^{(n)}-\E{\tau^{(n)}\vert \mathcal{Y}_{n}}
\right)}
+
\E{\left(U^{(n)}-\E{\tau^{(n)}\vert \mathcal{Y}_{n}}\right)^{2}} 
\\ \stackrel{Lemmata~\ref{lemUnUnmETaunYn},\ref{lemUnmETaunYn2}}{=}
H_{n,1}^{2}+H_{n,2} -\frac{4n}{n-1}\left(H_{n,1}+H_{n,2} -1-n^{-1}H_{n,1}^{2}\right)
+ \frac{4n^{2}H_{n,2}}{3(n-1)^{2}} 
\\+\frac{3n^{2}}{(n-1)^{2}} 
-\frac{8nH_{n,1}}{(n-1)^{2}} 
-\frac{8nH_{n,2}}{(n-1)^{2}} 
+\frac{34n}{3(n-1)^{2}} 
+\frac{4H_{n,1}^{2}}{(n-1)^{2}} 
-\frac{16H_{n,2}}{3(n-1)^{2}} 
+\frac{13}{3(n-1)^{2}} 
-\frac{8H_{n,1}}{3n(n-1)^{2}} 
\\=
H_{n,1}^{2}-4H_{n,1}- \frac{5H_{n,2}}{3} +7
+\frac{4H_{n,1}^{2}}{n}
-\frac{12H_{n,1}}{n-1}
-\frac{28H_{n,2}}{3(n-1)} 
+\frac{64}{3(n-1)}
+\frac{8H_{n,1}^{2}}{(n-1)^{2}} 
\\-\frac{8H_{n,1}}{(n-1)^{2}} 
-\frac{12H_{n,2}}{(n-1)^{2}} 
+\frac{56}{3(n-1)^{2}} 
-\frac{4H_{n,1}^{2}}{n(n-1)^{2}}
-\frac{8H_{n,1}}{3n(n-1)^{2}} .
\end{array}
$$
This equation is compared with simulations in Fig. \ref{fig:fig1Lem1_6}.
\end{proof}

\begin{lemma}\label{lemtaunmETaunYn2}
For a Yule tree with speciation rate $\lambda=1$ we have  
\be\label{eqtaunmETaunYn2}
\begin{array}{l}
\E{(\tau^{(n)}-\E{\tau^{(n)}\vert \mathcal{Y}_{n}})^{2}}=
\E{\left(\left(U^{(n)}-\tau^{(n)}\right)-\left(U^{(n)}-\E{\tau^{(n)}\vert \mathcal{Y}_{n}}\right)\right)^{2}} 
\\ = 
\frac{8}{3}H_{n,2}-3-\frac{2H_{n,1}^{2}}{(n-1)}
+ \Theta(n^{-1}\ln n)
\xrightarrow{n\to \infty} \frac{4\pi^{2}}{9} - 3 . 
\end{array}
\ee
\end{lemma}
\begin{proof}
We begin by rewriting
$$
\begin{array}{l}
\E{\left(\tau^{(n)}-\E{\tau^{(n)}\vert \mathcal{Y}_{n}}\right)^{2}} 
=
\E{\left(\left(U^{(n)}-\tau^{(n)}\right)-\left(U^{(n)}-\E{\tau^{(n)}\vert \mathcal{Y}_{n}}\right)\right)^{2}} 
\\ = \E{\left(U^{(n)}-\tau^{(n)}\right)^{2}}
- 2\E{\left(U^{(n)}-\tau^{(n)}\right)\left(U^{(n)}-\E{\tau^{(n)}\vert \mathcal{Y}_{n}}\right)}
+\E{\left(\E{U^{(n)}-\tau^{(n)}\vert \mathcal{Y}_{n}}\right)^{2}} 
\end{array}
$$
and we recall that we know 
$\E{\left(U^{(n)}-\tau^{(n)}\right)^{2}}$
and Lemma \ref{lemUnmETaunYn2} gives us $\E{\left(U^{(n)}-\E{\tau^{(n)}\vert \mathcal{Y}_{n}}\right)^{2}}$.
We then turn to considering the remaining

\be\label{eqEUnmTaunUnmETaunYn2}
\begin{array}{l}
\E{(\left(U^{(n)}-\tau^{(n)}\right)\left(U^{(n)}-\E{\tau^{(n)}\vert \mathcal{Y}_{n}}\right)}
= \E{\E{(\left(U^{(n)}-\tau^{(n)}\right)\left(U^{(n)}-\E{\tau^{(n)}\vert \mathcal{Y}_{n}}\right)}\vert \mathcal{Y}_{n}}
\\ = \E{\E{U^{(n)}-\tau^{(n)}}\vert \mathcal{Y}_{n}\E{U^{(n)}-\tau^{(n)}\vert \mathcal{Y}_{n}}}
= \E{\left(\E{U^{(n)}-\tau^{(n)}}\vert \mathcal{Y}_{n}\right)^{2}}
\end{array}
\ee
and hence can write 

\be\label{eqEUnmTaunUnmETaunYn2sq}
\E{\left(\left(U^{(n)}-\tau^{(n)}\right)-\left(U^{(n)}-\E{\tau^{(n)}\vert \mathcal{Y}_{n}}\right)\right)^{2}} 
= \E{\left(U^{(n)}-\tau^{(n)}\right)^{2}}
-\E{\left(\E{U^{(n)}-\tau^{(n)}\vert \mathcal{Y}_{n}}\right)^{2}} 
\ee
and plugging in our previously known values

$$
\begin{array}{l}
\E{\left(U^{(n)}-\tau^{(n)}\right)^{2}} -\E{\E{\left(U^{(n)}-\tau^{(n)}\vert \mathcal{Y}_{n}}\right)^{2}} 
\\ \stackrel{Lemmata~\ref{lemUnmTaun2},\ref{lemUnmETaunYn2}}{=}
\frac{4n+2}{n-1}H_{n,2}-\frac{2H_{n,1}^{2}}{n-1}-\frac{4H_{n,1}}{n-1}
-\frac{4n^{2}H_{n,2}}{3(n-1)^{2}} 
-\frac{3n^{2}}{(n-1)^{2}} 
+\frac{8nH_{n,1}}{(n-1)^{2}} 
+\frac{8nH_{n,2}}{(n-1)^{2}} 
\\-\frac{34n}{3(n-1)^{2}} 
-\frac{4H_{n,1}^{2}}{(n-1)^{2}} 
+\frac{16H_{n,2}}{3(n-1)^{2}} 
-\frac{13}{3(n-1)^{2}} 
+\frac{8H_{n,1}}{3n(n-1)^{2}} 
\\=
4H_{n,2}-\frac{4}{3}H_{n,2}
+\frac{10H_{n,2}}{3(n-1)} -\frac{4H_{n,2}}{3(n-1)^{2}} 
-3
-\frac{6}{(n-1)} 
-\frac{3}{(n-1)^{2}} 
-\frac{2H_{n,1}^{2}}{n-1}-\frac{4H_{n,1}^{2}}{(n-1)^{2}} 
+\frac{4H_{n,1}}{(n-1)} +\frac{8H_{n,1}}{(n-1)^{2}} 
\\ 
+\frac{8H_{n,2}}{n-1} +\frac{8H_{n,2}}{(n-1)^{2}} 
-\frac{34}{3(n-1)} -\frac{34}{3(n-1)^{2}} 
+\frac{16H_{n,2}}{3(n-1)^{2}} 
-\frac{13}{3(n-1)^{2}} 
+\frac{8H_{n,1}}{3n(n-1)^{2}} 
\\=
\frac{8}{3}H_{n,2}-3
-\frac{2H_{n,1}^{2}}{n-1}

+\frac{4H_{n,1}}{(n-1)}

+\frac{34H_{n,2}}{3(n-1)}
-\frac{52}{3(n-1)}

-\frac{4H_{n,1}^{2}}{(n-1)^{2}} 
+\frac{8H_{n,1}}{(n-1)^{2}} 
+\frac{12H_{n,2}}{(n-1)^{2}} 
-\frac{56}{3(n-1)^{2}} 

+\frac{8H_{n,1}}{3n(n-1)^{2}} 
\\ \xlongrightarrow{n\to \infty} \frac{4\pi^{2}}{6}-\frac{2\pi^{2}}{9}-3  = \frac{4\pi^{2}}{9} - 3\approx 1.38649 .
\end{array}
$$
This equation is compared with simulations in Fig. \ref{fig:fig2Lem7_11}.
\end{proof}

\begin{lemma}\label{lemvarTaun}
For a Yule tree with speciation rate $\lambda=1$ we have  
\be\label{eqvarTaun}
\begin{array}{l}
\var{\tau^{(n)}}
= \frac{n+1}{n-1}H_{n,2}
-\frac{2(n+1)}{(n-1)^{2}}H_{n,1}^{2}
+\frac{4(n+1)}{(n-1)^{2}}H_{n,1}
-\frac{4}{n-1}
-\frac{4}{(n-1)^{2}}
\xrightarrow{n\to \infty} \frac{\pi^{2}}{6}.
\end{array}
\ee
\end{lemma}
\begin{proof}
We use the moment results presented at the beginning of this Section.
$$
\begin{array}{l}
\var{\tau^{(n)}}
= \E{\tau^{(n)^{2}}}
-\left(\E{\tau^{(n)}} \right)^{2}
 =
\frac{n+1}{n-1}\left((H_{n,1}-2)^{2}+H_{n,2} \right)-\frac{4}{n-1}
-\left(\frac{n+1}{n-1}H_{n,1}-\frac{2n}{n-1}\right)^{2}
\\=
\frac{n+1}{n-1}\left(H_{n,1}^{2}-4H_{n,1} + 4+H_{n,2} \right)
-\frac{4}{n-1}
-\left(\frac{n+1}{n-1}\right)^{2}H_{n,1}^{2}
+\frac{4n(n+1)}{(n-1)^{2}}H_{n,1}
-\frac{4n^{2}}{(n-1)^{2}}
\\=
\frac{n+1}{n-1}H_{n,2}
+\frac{n+1}{n-1}(1-\frac{n+1}{n-1})H_{n,1}^{2}
-4\frac{n+1}{n-1}(1-\frac{n}{n-1})H_{n,1}
+4(\frac{n+1}{n-1}-\frac{n^{2}}{(n-1)^{2}})
-\frac{4}{n-1}
\\=
\frac{n+1}{n-1}H_{n,2}
-\frac{2(n+1)}{(n-1)^{2}}H_{n,1}^{2}
+\frac{4(n+1)}{(n-1)^{2}}H_{n,1}
+4\frac{n^{2}-1-n^{2}}{(n-1)^{2}}
-\frac{4}{n-1}
\\=
\frac{n+1}{n-1}H_{n,2}
-\frac{2(n+1)}{(n-1)^{2}}H_{n,1}^{2}
+\frac{4(n+1)}{(n-1)^{2}}H_{n,1}
-\frac{4}{n-1}
-\frac{4}{(n-1)^{2}}
\xrightarrow{n\to \infty} \frac{\pi^{2}}{6}
\approx 1.645.
\end{array}
$$
This equation is compared with simulations in Fig. \ref{fig:fig2Lem7_11}.
\end{proof}

\begin{lemma}\label{lemvarTaunYn}
For a Yule tree with speciation rate $\lambda=1$ we have  
\be\label{eqvarTaunYn}
\begin{array}{l}
\var{\E{\tau^{(n)}\vert  \mathcal{Y}_{n}}}
= 
3 - \frac{5}{3} H_{n,2}
-\frac{28H_{n,2}}{3(n-1)} +\frac{40}{3(n-1)}
+\Theta(n^{-2}) 
\xrightarrow{n\to \infty} 3- \frac{5\pi^{2}}{18}.
\end{array}
\ee
\end{lemma}
\begin{proof}
We use the moment results presented at the beginning of this Section.
$$
\begin{array}{l}
\var{\E{\tau^{(n)}\vert  \mathcal{Y}_{n}}}
= \E{\E{\tau^{(n)}\vert  \mathcal{Y}_{n}}^{2}}
-\left(\E{\tau^{(n)}} \right)^{2}
\\
\stackrel{Lemma ~\ref{lemETaunYn2}~\mathrm{and}~\text{\cite{SSagKBar2012art}}}{=}
\\
H_{n,1}^{2}-4H_{n,1}- \frac{5H_{n,2}}{3} +7
+\frac{4H_{n,1}^{2}}{n}
-\frac{12H_{n,1}}{n-1}
-\frac{28H_{n,2}}{3(n-1)} 
+\frac{64}{3(n-1)}
+\frac{8H_{n,1}^{2}}{(n-1)^{2}} 
\\-\frac{8H_{n,1}}{(n-1)^{2}} 
-\frac{12H_{n,2}}{(n-1)^{2}} 
+\frac{56}{3(n-1)^{2}} 
-\frac{4H_{n,1}^{2}}{n(n-1)^{2}}
-\frac{8H_{n,1}}{3n(n-1)^{2}}
-\left(H_{n,1}+\frac{2H_{n,1}}{n-1}-\frac{2n}{n-1} \right)^{2}
\end{array}
$$

$$
\begin{array}{l}
=H_{n,1}^{2}-4H_{n,1}- \frac{5H_{n,2}}{3} +7
+\frac{4H_{n,1}^{2}}{n}
-\frac{12H_{n,1}}{n-1}
-\frac{28H_{n,2}}{3(n-1)} 
+\frac{64}{3(n-1)}
+\frac{8H_{n,1}^{2}}{(n-1)^{2}} 
\\-\frac{8H_{n,1}}{(n-1)^{2}} 
-\frac{12H_{n,2}}{(n-1)^{2}} 
+\frac{56}{3(n-1)^{2}} 
-\frac{4H_{n,1}^{2}}{n(n-1)^{2}}
-\frac{8H_{n,1}}{3n(n-1)^{2}}
\\
-\left(
H_{n,1}^{2}
+\frac{4H_{n,1}^{2}}{n-1}
-4H_{n,1}
-\frac{4H_{n,1}}{n-1}
+\frac{4H_{n,1}^{2}}{(n-1)^{2}}
-\frac{8nH_{n,1}}{(n-1)^{2}} 
+\frac{4n^{2}}{(n-1)^{2}} 
\right)
\end{array}
$$

$$
\begin{array}{l}
=
3 - \frac{5H_{n,2}}{3} 
-\frac{28H_{n,2}}{3(n-1)} 
+\frac{40}{3(n-1)}
-\frac{12H_{n,2}}{(n-1)^{2}} 
+\frac{44}{3(n-1)^{2}}
-\frac{8H_{n,1}}{3n(n-1)^{2}} 
\xrightarrow{n\to \infty} 3- \frac{5\pi^{2}}{18}
\approx 0.258.
\end{array}
$$
This equation is compared with simulations in Fig. \ref{fig:fig2Lem7_11}.
\end{proof}

\begin{lemma}\label{lemvarUnmTaun}
For a Yule tree with speciation rate $\lambda=1$ we have  
\be\label{eqvarUnmTaun}
\begin{array}{rcl}
\var{U^{(n)}-\tau^{(n)}}
&= &4(H_{n,2}-1)-\frac{2H_{n,1}^{2}}{n-1}
+\frac{4H_{n,1}}{n-1}
+\frac{6H_{n,2}}{n-1} -\frac{8}{n-1}
-\frac{4(H_{n,1}-1)^{2}}{(n-1)^{2}}
\\ &\xrightarrow{n\to \infty} &\frac{2\pi^{2}}{3}-4 .
\end{array}
\ee
\end{lemma}
\begin{proof}
We use the moment results presented at the beginning of this Section.
$$
\begin{array}{l}
\var{U^{(n)}-\tau^{(n)}}
= \E{\left(U^{(n)}-\tau^{(n)}\right)^{2}}
-\left(\E{U^{(n)}-\tau^{(n)}} \right)^{2}
\\ =
\frac{4n+2}{n-1}H_{n,2} -\frac{2H_{n,1}^{2}}{n-1}-\frac{4H_{n,1}}{n-1}
-\left(2(n-H_{n,1})/(n-1) \right)^{2}
\\ =
\frac{4n+2}{n-1}H_{n,2} -\frac{2H_{n,1}^{2}}{n-1}-\frac{4H_{n,1}}{n-1}
-4+\frac{8(H_{n,1}-1)}{n-1}-\frac{4(H_{n,1}-1)^{2}}{(n-1)^{2}}
\\ =
4(H_{n,2}-1)+\frac{6H_{n,2}}{n-1} -\frac{2H_{n,1}^{2}}{n-1}
+\frac{4H_{n,1}-8}{n-1}-\frac{4(H_{n,1}-1)^{2}}{(n-1)^{2}}
\xrightarrow{n\to \infty} \frac{2\pi^{2}}{3}-4
\approx 2.5797.
\end{array}
$$
This equation is compared with simulations in Fig. \ref{fig:fig2Lem7_11}.
\end{proof}

\begin{lemma}\label{lemvarUnmETaunYn}
For a Yule tree with speciation rate $\lambda=1$ we have  
\be\label{eqvarUnmETaunYn}
\begin{array}{l}
\var{U^{(n)}-\E{\tau^{(n)}\vert \mathcal{Y}_{n}}}
= 
\frac{4H_{n,2}}{3} -1
-\frac{16H_{n,2}}{3(n-1)} +\frac{28}{3(n-1)} + \Theta(n^{-2})
\xrightarrow{n\to \infty} \frac{2\pi^{2}}{9}-1
.
\end{array}
\ee
\end{lemma}
\begin{proof}
We begin by rewriting
$$
\begin{array}{l}
\var{U^{(n)}-\E{\tau^{(n)}\vert \mathcal{Y}_{n}}} =
\E{\left(U^{(n)}-\E{\tau^{(n)}\vert \mathcal{Y}_{n}}\right)^{2}}
-\left(\E{U^{(n)}-\E{\tau^{(n)}\vert \mathcal{Y}_{n}}}\right)^{2}
\\ \stackrel{Lemma~\ref{lemUnmETaunYn2},~\text{\cite{SSagKBar2012art}}}{=}
\\
\frac{4n^{3}H_{n,2}}{3n(n-1)^{2}} 
+\frac{9n^{3}}{3n(n-1)^{2}} 
-\frac{24n^{2}H_{n,1}}{3n(n-1)^{2}} 
-\frac{24n^{2}H_{n,2}}{3n(n-1)^{2}} 
+\frac{34n^{2}}{3n(n-1)^{2}} 
+\frac{12nH_{n,1}^{2}}{3n(n-1)^{2}} 
-\frac{16nH_{n,2}}{3n(n-1)^{2}} 
\\+\frac{13n}{3n(n-1)^{2}} 
-\frac{8H_{n,1}}{3n(n-1)^{2}} 
-\left(\frac{2(n-H_{n,1})}{(n-1)} \right)^{2}
\end{array}
$$
$$
\begin{array}{l}
= 
\frac{4H_{n,2}}{3} -1
-\frac{16H_{n,2}}{3(n-1)} 
+\frac{28}{3(n-1)} 
-\frac{12H_{n,2}}{(n-1)^{2}} 
+\frac{44}{3(n-1)^{2}} 
-\frac{8H_{n,1}}{3n(n-1)^{2}} 
\xrightarrow{n\to \infty} \frac{2\pi^{2}}{9}-1
\approx 1.193.
\end{array}
$$
This equation is compared with simulations in Fig. \ref{fig:fig2Lem7_11}.
\end{proof}

\begin{lemma}\label{lemcovTaunETaunYn}
For a Yule tree with speciation rate $\lambda=1$ we have  
\be\label{eqcovTaunETaunYn}
\begin{array}{l}
\cov{\tau^{(n)}}{\E{\tau^{(n)}\vert \mathcal{Y}_{n}}}
=
3 - \frac{5}{3} H_{n,2}-\frac{28H_{n,2}}{3(n-1)} +\Theta(n^{-1}) 
\xrightarrow{n\to \infty} 3- \frac{5\pi^{2}}{18},
\\
\corr{\tau^{(n)}}{\E{\tau^{(n)}\vert \mathcal{Y}_{n}}}
= 
\sqrt{\frac{3}{H_{n,2}}-\frac{5}{3}}+\Theta(n^{-1})
\xrightarrow{n\to \infty} 
\sqrt{\frac{18}{\pi^{2}}-\frac{5}{3}}.
\end{array}
\ee
\end{lemma}
\begin{proof}
$$
\begin{array}{l}
\cov{\tau^{(n)}}{\E{\tau^{(n)}\vert \mathcal{Y}_{n}}}
=\E{\tau^{(n)}\E{\tau^{(n)}\vert \mathcal{Y}_{n}}}-
\E{\tau^{(n)}}\E{\E{\tau^{(n)}\vert \mathcal{Y}_{n}}}
\\=\E{\E{\tau^{(n)}\E{\tau^{(n)}\vert \mathcal{Y}_{n}}\vert \mathcal{Y}_{n}}}-\left(\E{\tau^{(n)}}\right)^{2}
= \E{\E{\tau^{(n)}\vert \mathcal{Y}_{n}}\E{\tau^{(n)}\vert \mathcal{Y}_{n}}}-\left(\E{\tau^{(n)}}\right)^{2}
\\= \E{\left(\E{\tau^{(n)}\vert \mathcal{Y}_{n}}\right)^{2}}-\left(\E{\tau^{(n)}}\right)^{2}
\end{array}
$$
$$
\begin{array}{l}
\stackrel{Lemma~ \ref{lemETaunYn2},~\text{\cite{SSagKBar2012art}}}{=}
\\ 
H_{n,1}^{2}-4H_{n,1}- \frac{5H_{n,2}}{3} +7
+\frac{4H_{n,1}^{2}}{n}
-\frac{12H_{n,1}}{n-1}
-\frac{28H_{n,2}}{3(n-1)} 
+\frac{64}{3(n-1)}
+\frac{8H_{n,1}^{2}}{(n-1)^{2}} 
\\-\frac{8H_{n,1}}{(n-1)^{2}} 
-\frac{12H_{n,2}}{(n-1)^{2}} 
+\frac{56}{3(n-1)^{2}} 
-\frac{4H_{n,1}^{2}}{n(n-1)^{2}}
-\frac{8H_{n,1}}{3n(n-1)^{2}} 
-\left(\frac{n+1}{n-1}H_{n,1}-\frac{2n}{n-1} \right)^{2}
\end{array}
$$

$$
\begin{array}{l}
=
3-\frac{5H_{n,2}}{3}
-\frac{28H_{n,2}}{3(n-1)} 
+\frac{40}{3(n-1)}
-\frac{12H_{n,2}}{(n-1)^{2}} 
+\frac{44}{3(n-1)^{2}}
-\frac{8H_{n,1}}{3n(n-1)^{2}} 
\xrightarrow{n\to \infty} 3- \frac{5\pi^{2}}{18}
\approx 0.258.
\end{array}
$$

$$
\begin{array}{l}
\corr{\tau^{(n)}}{\E{\tau^{(n)}\vert \mathcal{Y}_{n}}}
=
\frac{\cov{\tau^{(n)}}{\E{\tau^{(n)}\vert \mathcal{Y}_{n}}}}
{\sqrt{(\var{\tau^{(n)}})(\var{\E{\tau^{(n)}\vert \mathcal{Y}_{n}}})}}
\\ \stackrel{Lemmata~\ref{lemvarTaun},\ref{lemvarTaunYn}}{=}
\frac{3 - \frac{5}{3} H_{n,2} +\Theta(n^{-1})}
{\sqrt{(H_{n,2}+\Theta(n^{-1}\ln^{2}n))
(3-\frac{5}{3}H_{n,2}+\Theta(n^{-1}))}}
=
\frac{3 - \frac{5}{3} H_{n,2}}{\sqrt{H_{n,2}(3 - \frac{5}{3} H_{n,2})}}
+\Theta(n^{-1})
 \\=\sqrt{\frac{3}{H_{n,2}}-\frac{5}{3}}+\Theta(n^{-1})
\xrightarrow{n\to \infty} 
\sqrt{\frac{18}{\pi^{2}}-\frac{5}{3}}
\approx 0.396.
\end{array}
$$
These equations are compared with simulations in Fig. \ref{fig:fig3Lem12_14}.
\end{proof}

\begin{lemma}\label{lemcovUnmTaunUnmETaunYn}
For a Yule tree with speciation rate $\lambda=1$ we have  
\be\label{eqcovUnmTaunUnmETaunYn}
\begin{array}{l}
\cov{U^{(n)}-\tau^{(n)}}{U^{(n)}-\E{\tau^{(n)}\vert \mathcal{Y}_{n}}}
\frac{4H_{n,2}}{3} -1 +\Theta(n^{-1})
\xrightarrow{n\to \infty} 
\frac{2\pi^{2}}{9}-1,
\\
\corr{U^{(n)}-\tau^{(n)}}{U^{(n)}-\E{\tau^{(n)}\vert \mathcal{Y}_{n}}}
= \sqrt{
\frac{\frac{4}{3}H_{n,2} -1}{
4H_{n,2}-4
}}
+\Theta(n^{-1})
\xrightarrow{n\to \infty}  
\sqrt{
\frac{\frac{\pi^{2}}{12} -1}{
\frac{2\pi^{2}}{3}-4
}}.
\end{array}
\ee
\\
\end{lemma}
\begin{proof}
We will use the law of total covariance
$$
\begin{array}{l}
\cov{U^{(n)}-\tau^{(n)}}{U^{(n)}-\E{\tau^{(n)}\vert \mathcal{Y}_{n}}}
\\=
\cov{\E{U^{(n)}-\tau^{(n)}\vert \mathcal{Y}_{n}}}{\E{U^{(n)}-\E{\tau^{(n)}\vert \mathcal{Y}_{n}}\vert \mathcal{Y}_{n}}}
\\+
\E{\cov{U^{(n)}-\tau^{(n)}}{\left(U^{(n)}-\E{\tau^{(n)}\vert \mathcal{Y}_{n}}\right)\vert \mathcal{Y}_{n}}}
\\=  
\var{U^{(n)}-\E{\tau^{(n)}\vert \mathcal{Y}_{n}}} +
\E{\E{\left(U^{(n)}-\tau^{(n)}\right)\left(U^{(n)}-\E{\tau^{(n)}\vert \mathcal{Y}_{n}}\right)\vert \mathcal{Y}_{n}}}
\\
-
\E{\E{U^{(n)}-\tau^{(n)}\vert \mathcal{Y}_{n}}
\E{\left(U^{(n)}-\E{\tau^{(n)}\vert \mathcal{Y}_{n}}\right)\vert \mathcal{Y}_{n}}}
\\=  
\var{U^{(n)}-\E{\tau^{(n)}\vert \mathcal{Y}_{n}}} +
\E{\left(\E{U^{(n)}-\tau^{(n)}\vert \mathcal{Y}_{n}}\right)^{2}}-\E{\left(\E{U^{(n)}-\tau^{(n)}\vert \mathcal{Y}_{n}}\right)^{2}}
\\
= \var{U^{(n)}-\E{\tau^{(n)}\vert \mathcal{Y}_{n}}} 
\stackrel{Lemma~\ref{lemvarUnmETaunYn}}{=}
\frac{4H_{n,2}}{3} -1
-\frac{16H_{n,2}}{3(n-1)} +\frac{28}{3(n-1)}  \Theta(n^{-2})
\\ \xrightarrow{n\to \infty} \frac{2\pi^{2}}{9}-1
\approx 1.193.
\end{array}
$$
Then using the above, Lemmata \ref{lemvarUnmETaunYn}, and \ref{lemvarUnmTaun} we can write
$$
\begin{array}{l}
\corr{U^{(n)}-\tau^{(n)}}{U^{(n)}-\E{\tau^{(n)}\vert \mathcal{Y}_{n}}}
=
\frac{\cov{U^{(n)}-\tau^{(n)}}{U^{(n)}-\E{\tau^{(n)}\vert \mathcal{Y}_{n}}}}
{\sqrt{\var{U^{(n)}-\tau^{(n)}}}\sqrt{\var{U^{(n)}-\E{\tau^{(n)}\vert \mathcal{Y}_{n}}}}}
\\ =\frac{\var{U^{(n)}-\E{\tau^{(n)}\vert \mathcal{Y}_{n}}}}{\sqrt{\var{U^{(n)}-\tau^{(n)}}}\sqrt{\var{U^{(n)}-\E{\tau^{(n)}\vert\mathcal{Y}_{n}}}}}
=
\frac{\frac{4H_{n,2}}{3} -1 + \Theta(n^{-1})}{
\sqrt{4H_{n,2}-4+\Theta(n^{-1}\ln^{2}n)}
\sqrt{\frac{4H_{n,2}}{3} -1 + \Theta(n^{-1})
}}
\\=
\sqrt{
\frac{\frac{4}{3}H_{n,2} -1}{
4H_{n,2}-4
}}
+\Theta(n^{-1})
\xrightarrow{n\to \infty}  
\sqrt{
\frac{\frac{2*\pi^{2}}{9} -1}{
\frac{2\pi^{2}}{3}-4
}}
\approx 0.680.
\end{array}
$$
These equations are compared with simulations in Fig. \ref{fig:fig3Lem12_14}.
\end{proof}

\begin{lemma}\label{lemcovUnmTaunTaunmETaunYn}
For a Yule tree with speciation rate $\lambda=1$ we have  
\be\label{eqcovUnmTaunTaunmETaunYn}
\begin{array}{l}
\cov{U^{(n)}-\tau^{(n)}}{\tau^{(n)}-\E{\tau^{(n)}\vert \mathcal{Y}_{n}}}
= 
3-\frac{8}{3}H_{n,2} + \Theta(n^{-1}\ln^{2} n) \xrightarrow{n\to \infty} 3-\frac{4\pi^{2}}{9},
\\
\corr{U^{(n)}-\tau^{(n)}}{\tau^{(n)}-\E{\tau^{(n)}\vert \mathcal{Y}_{n}}}
=
-\sqrt{\frac{\frac{8}{3}H_{n,2}-3}{4(H_{n,2}-1)}} + \Theta(n^{-1}\ln^{2} n)
\xrightarrow{n\to \infty}
-\sqrt{\frac{\frac{4\pi^{2}}{9}-3}{\frac{2\pi^{2}}{3}-4}}.
\end{array}
\ee
\end{lemma}
\begin{proof}
We will use the law of total covariance
$$
\begin{array}{l}
\cov{U^{(n)}-\tau^{(n)}}{\tau^{(n)}-\E{\tau^{(n)}\vert \mathcal{Y}_{n}}}
=
\E{\cov{U^{(n)}-\tau^{(n)}}{\left(\tau^{(n)}-\E{\tau^{(n)}\vert \mathcal{Y}_{n}}\right)\vert \mathcal{Y}_{n}}}
\\+
\cov{\E{U^{(n)}-\tau^{(n)}\vert \mathcal{Y}_{n}}}{\E{\tau^{(n)}-\E{\tau^{(n)}\vert \mathcal{Y}_{n}}\vert \mathcal{Y}_{n}}}
\\=  
\E{\E{\left(U^{(n)}-\tau^{(n)}\right)\left(\tau^{(n)}-\E{\tau^{(n)}\vert \mathcal{Y}_{n}}\right)\vert \mathcal{Y}_{n}}}
\\-
\E{
\E{\left(U^{(n)}-\tau^{(n)} \right)\vert \mathcal{Y}_{n} }
\E{\left(\tau^{(n)}-\E{\tau^{(n)}\vert \mathcal{Y}_{n}}\right)\vert \mathcal{Y}_{n}}
} + 0
\\=  
\E{\left(U^{(n)}-\tau^{(n)}\right)\left(\tau^{(n)}-\E{\tau^{(n)}\vert \mathcal{Y}_{n}}\right)}-0
\\=
\E{\left(U^{(n)}-\tau^{(n)}\right)\left(\pm U^{(n)}+\tau^{(n)}-\E{\tau^{(n)}\vert \mathcal{Y}_{n}}\right)}
\\=
-\E{\left(U^{(n)}-\tau^{(n)}\right)^{2}}+
\E{\left(U^{(n)}-\tau^{(n)}\right)\left(U^{(n)}-\E{\tau^{(n)}\vert \mathcal{Y}_{n}}\right)}
\\ \stackrel{Eq.\eqref{eqEUnmTaunUnmETaunYn2}}{=}
\E{\left(\E{U^{(n)}-\tau^{(n)}}\vert \mathcal{Y}_{n}\right)^{2}}
-\E{\left(U^{(n)}-\tau^{(n)}\right)^{2}}
\\ \stackrel{Eq.\eqref{eqEUnmTaunUnmETaunYn2sq}}{=}
-\E{(\tau^{(n)}-\E{\tau^{(n)}\vert \mathcal{Y}_{n}})^{2}}
\\
\stackrel{Lemma~\ref{lemtaunmETaunYn2}}{=}
3-\frac{8}{3}H_{n,2} + \Theta(n^{-1}\ln^{2} n)
\xrightarrow{n\to \infty} 3-\frac{4\pi^{2}}{9} \approx -1.38649 .
\end{array}
$$
Then using the above, Lemmata \ref{lemtaunmETaunYn2} and \ref{lemvarUnmTaun}, and remembering that $\tau^{(n)}-\E{\tau^{(n)}}$ has mean $0$, we can write
$$
\begin{array}{l}
\corr{U^{(n)}-\tau^{(n)}}{\tau^{(n)}-\E{\tau^{(n)}\vert \mathcal{Y}_{n}}}
=
\frac{\cov{U^{(n)}-\tau^{(n)}}{\tau^{(n)}-\E{\tau^{(n)}\vert \mathcal{Y}_{n}}}}
{\sqrt{\var{U^{(n)}-\tau^{(n)}}}{\var{\tau^{(n)}-\E{\tau^{(n)}\vert \mathcal{Y}_{n}}}}}
\\ =
\frac{3-\frac{8}{3}H_{n,2} + \Theta(n^{-1}\ln^{2} n)}{\sqrt{(4(H_{n,2}-1)+\Theta(n^{-1}\ln^{2} n))(\frac{8}{3}H_{n,2}-3+ \Theta(n^{-1}\ln^{2} n))}}
=
\frac{3-\frac{8}{3}H_{n,2}}{\sqrt{(4(H_{n,2}-1))(\frac{8}{3}H_{n,2}-3 )}}
+ \Theta(n^{-1}\ln^{2} n)
\\ =
-\sqrt{\frac{\frac{8}{3}H_{n,2}-3}{4(H_{n,2}-1)}} + \Theta(n^{-1}\ln^{2} n)
\xrightarrow{n\to \infty}
-\sqrt{\frac{\frac{4\pi^{2}}{9}-3}{\frac{2\pi^{2}}{3}-4}} \approx -0.733.
\end{array}
$$
These equations are compared with simulations in Fig. \ref{fig:fig3Lem12_14}.
\end{proof}

\section{Comparison with simulations}\label{secNumericalSimul}
\renewcommand{\theequation}{\thesection.\arabic{equation}}
\renewcommand{\thelemma}{\thesection.\arabic{lemma}}
\renewcommand{\thetheorem}{\thesection.\arabic{theorem}}
\renewcommand{\theremark}{\thesection.\arabic{remark}}
\renewcommand{\theexample}{\thesection.\arabic{example}}
\renewcommand{\thecorollary}{\thesection.\arabic{corollary}}
\renewcommand{\thealgorithm}{\thesection.\arabic{algorithm}}
\setcounter{equation}{0}
\setcounter{lemma}{0}
\setcounter{theorem}{0}
\setcounter{example}{0}
\setcounter{remark}{0}
\setcounter{corollary}{0}
\setcounter{algorithm}{0}

We have 
verified Lemmata \ref{lemUnmTaun2}-\ref{lemcovUnmTaunTaunmETaunYn} by simulations. We sampled pure-birth trees of increasing number of tips using the \texttt{R} package \texttt{TreeSim} \citep{TreeSim2}. For each tree, $U^{(n)}$ is obtained by the summation of $T_1, \cdots, T_n$. 
Using knowledge of a tree's topology, we can calculate $\E{1_{k}^{(n)}\vert \mathcal{Y}_{n}}$, 
which counts the fraction of pairs of nodes that coalesced at the $k$--th speciation event. We can use these proportions to calculate
\[\E{\tau^{(n)} \vert\mathcal{Y}_n} 
=\sum_{k=1}^{n-1}\E{1_{k}^{(n)}\vert \mathcal{Y}_{n}}\sum_{j=k+1}^{n}T_j.\]
The detailed procedure of sampling $\E{\tau^{(n)} \vert\mathcal{Y}_n}$ is explained in Algorithm \ref{simultaun}.
Furthermore, $\tau^{(n)}$ for an already simulated tree can be sampled by first sampling $\kappa_n$ from a categorical distribution with probabilities
\[P(\kappa_n = k) = \E{1_{k}^{(n)}\vert \mathcal{Y}_{n}},\]
for $k = 1, \cdots, n-1$, and then calculating
\[\tau^{(n)} = \sum_{j=\kappa_{n+1}}^{n}T_j.\]
The sampled values $U^{(n)}$, $E\left[\tau^{(n)}|\mathcal{Y}_n\right]$, and $\tau^{(n)}$ can then be used to calculate the left-hand sides of the Lemmata \ref{lemUnmTaun2}-\ref{lemcovUnmTaunTaunmETaunYn}. 

We compare our simulated values of the lemmata with the theoretical functions, i.e., the right-hand sides, which are shown in Figs. \ref{fig:fig1Lem1_6}--\ref{fig:fig3Lem12_14}. We present two cases in each lemma, which shows values for lower number of tips, from 2 to 10 and higher number of tips, which are 25, 100, 250, 500, 750, 1000, 1250, 1500, 1750, 2000, 2250, and 2500.

Figure \ref{fig:fig1Lem1_6} shows the values obtained from the simulation procedures and the theoretical values of Lemmata \ref{lemUnmTaun2} to \ref{lemETaunYn2}. The simulated values follow the theoretical lines closely in the case of low and large number of tips. In particular, Lemma \ref{lemUnmETaunYn2} is converging into $\frac{2\pi^2}{9} + 3 \approx 5.19325$ which is shown by the horizontal blue line.

In Fig. \ref{fig:fig2Lem7_11}, Lemmata \ref{lemtaunmETaunYn2} to \ref{lemvarUnmETaunYn} have an asymptotic value indicated by the horizontal blue line. As the number of tips increases, the theoretical curves of Lemmata \ref{lemtaunmETaunYn2} to \ref{lemvarUnmETaunYn} gets closer to the asymptotic value. The simulated values also follow the theoretical line, and hence the asymptotic value. 

In Fig. \ref{fig:fig3Lem12_14}, both calculated covariance and correlation values of Lemmata \ref{lemcovTaunETaunYn}--\ref{lemcovUnmTaunTaunmETaunYn} are presented. In this case, negative values are allowed and present for the covariance and correlation of Lemma \ref{lemcovUnmTaunTaunmETaunYn}. Figure \ref{fig:fig3Lem12_14} shows that the theoretical function values of Lemmata \ref{lemcovTaunETaunYn}--\ref{lemcovUnmTaunTaunmETaunYn} converge to an asymptotic value, and the simulated values comply with the theoretical curves.

One thing to note is that for the correlation in Lemma \ref{lemcovUnmTaunTaunmETaunYn}, the value for $n=2$ is not defined because one of the terms in the denominator of the formula for $\corr{U^{(n)}-\tau^{(n)}}{\tau^{(n)}-\E{\tau^{(n)}\vert \mathcal{Y}_{n}}}$, $\var{\tau^{(n)}-\E{\tau^{(n)}\vert \mathcal{Y}_{n}}}$, is zero when $n=2$. This is due to the fact that 
\[\tau^{(2)} = T_2,\] 
when $n=2$. At the same time, because there is only one coalescence event for a tree with two tips, then
\[\E{\tau^{(2)} \vert\mathcal{Y}_2} = \E{1_{1}^{(2)}\vert \mathcal{Y}_{2}}T_2 = 1 \cdot T_2 = T_2,\]
which leads to
\[\var{\tau^{(2)}-\E{\tau^{(2)}\vert \mathcal{Y}_{2}}} = \var{0} = 0.\]

\begin{figure}
    \centering
    \includegraphics[width=1.1\textwidth]{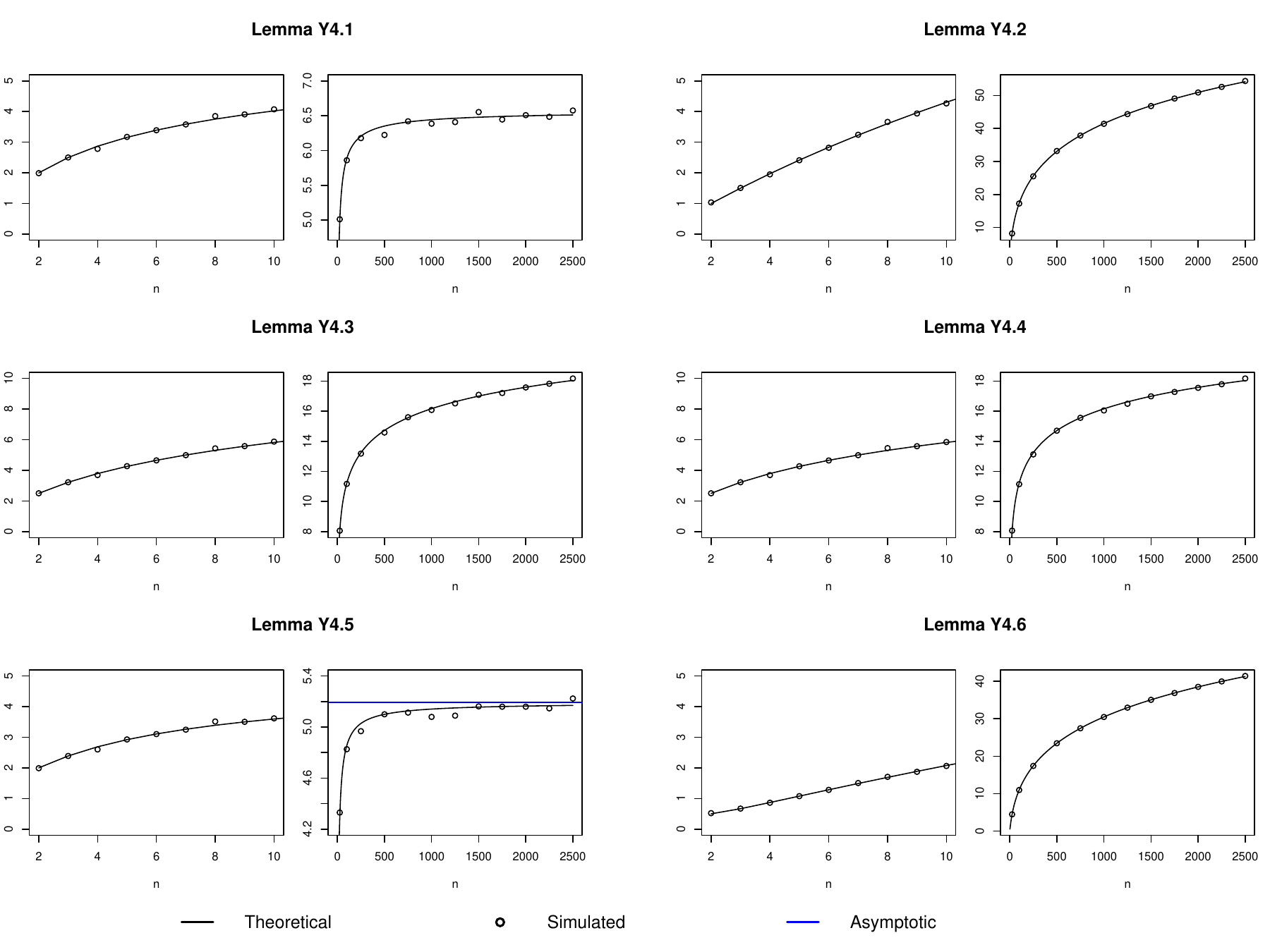}
    \caption{Comparison of simulated and theoretical values of Lemmata \ref{lemUnmTaun2} to \ref{lemETaunYn2}. For each lemma, the left-hand side plot shows values between 2 to 10, while the right-hand side plot shows values between 25 to 2500.}
    \label{fig:fig1Lem1_6}
\end{figure}

\begin{figure}
    \centering
    \includegraphics[width=1.1\textwidth]{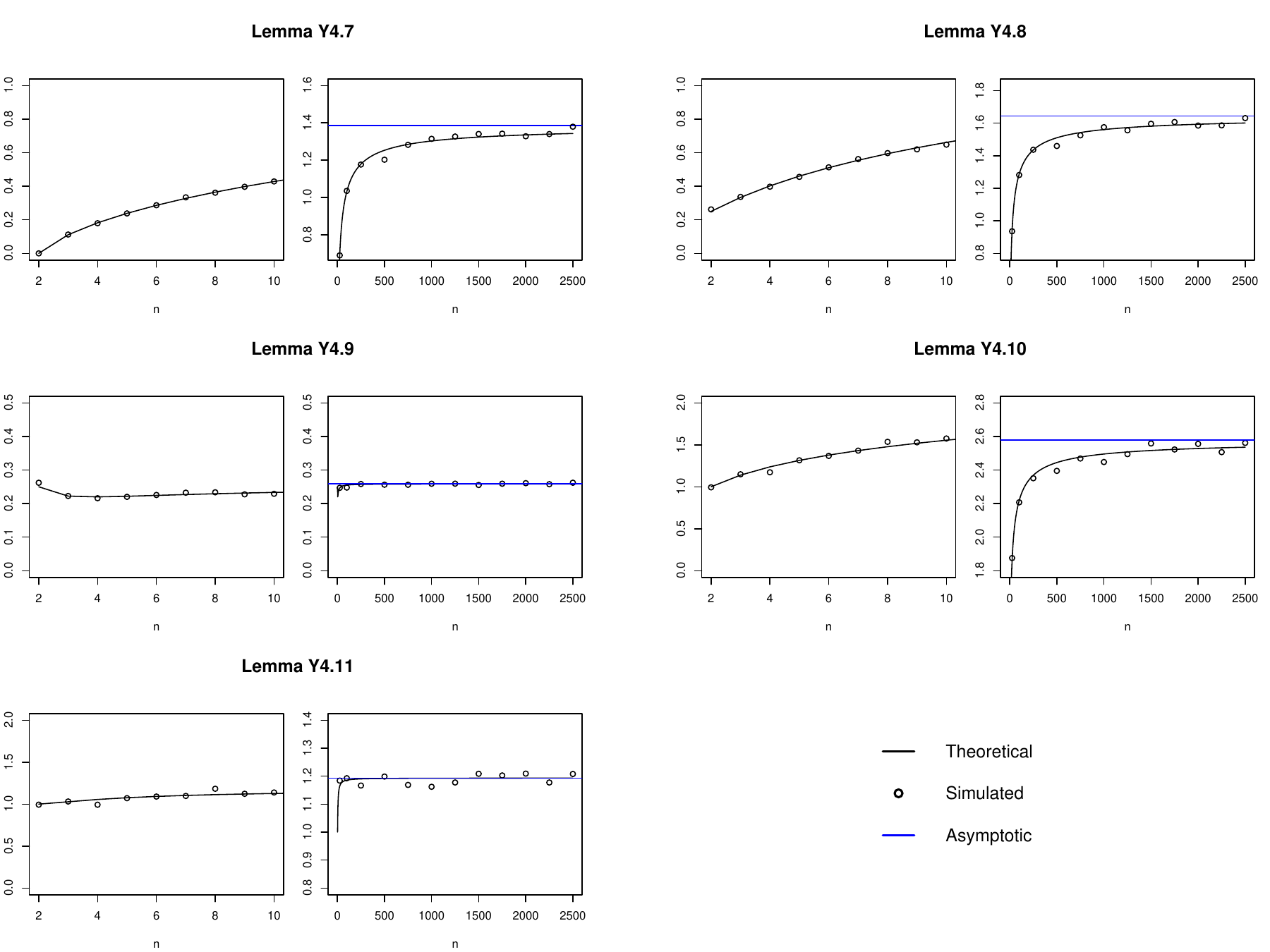}
    \caption{Comparison of simulated and theoretical values of Lemmata \ref{lemtaunmETaunYn2} to \ref{lemvarUnmETaunYn}. For each lemma, the left-hand side plot shows values between 2 to 10, while the right-hand side plot shows values between 25 to 2500.}
    \label{fig:fig2Lem7_11}
\end{figure}

\begin{figure}
    \centering
    \includegraphics[width=1.1\textwidth]{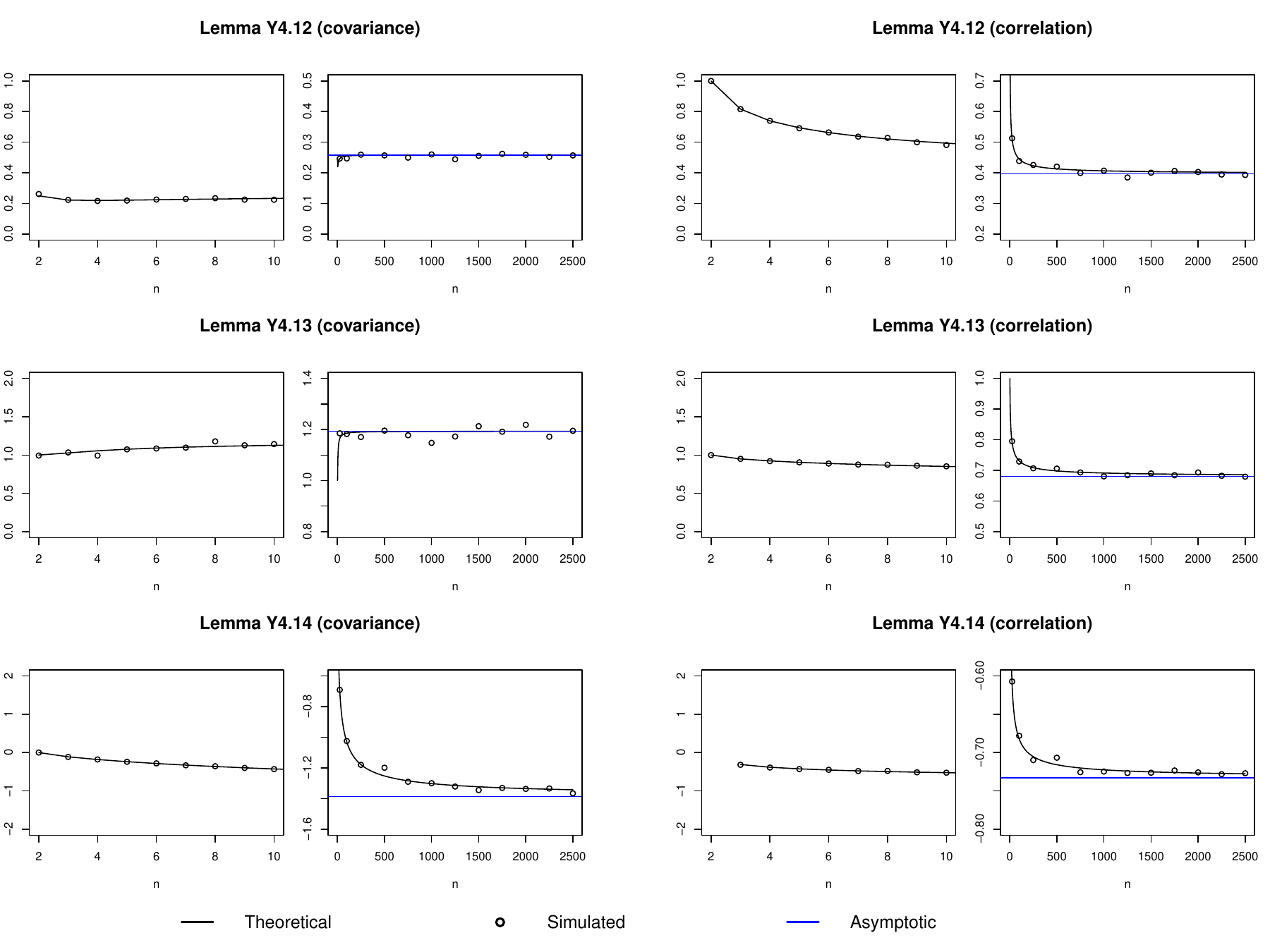}
    \caption{Comparison of simulated and theoretical values of Lemmata \ref{lemcovTaunETaunYn} to
    \ref{lemcovUnmTaunTaunmETaunYn}, where each lemma has a covariance and correlation part. For each lemma, the left-hand side plot shows values between 2 to 10, while the right-hand side plot shows values between 25 to 2500.}
    \label{fig:fig3Lem12_14}
\end{figure}

\begin{algorithm}
\caption{Sampling $\E{\tau^{(n)} \vert\mathcal{Y}_n}$}\label{simultaun}
\begin{algorithmic}[1]
\Require Number of tips, $n$.
\Ensure $\E{\tau^{(n)} \vert\mathcal{Y}_n}$
\State Generate a phylogenetic tree with $n$ number of tips, $\mathcal{Y}_n$, with internal nodes $z_k$, $k \in \{1, \cdots, n-1\}$, denoting the speciation points, and $T_1, \cdots, T_n$ denoting the time intervals between speciation times.
\For{$k = 1, \cdots, n-1$}
    \State Compute $\text{ntips}(z_k)$, number of tips of a subtree starting from $z_k$ (the $k$-th speciation event).
\EndFor
\For{$k = 1, \cdots, n-1$}
    \If{$\text{ntips}(z_k) > 2$}
        \State Set $l_{z_k}$, child node of $z_k$ on the left side.
        \State Set $r_{z_k}$, child node of $z_k$ on the right side.
        \State Compute $s_k := \text{ntips}(l_{z_k}) \times \text{ntips}(r_{z_k})$, number of possible pairs of tip nodes that coalesced at the $k$-th speciation event.
    \Else 
        \If{$\text{ntips}(z_k) == 2$}
            \State $s_k := 2$ ($z_k$ is a node above two tips)
        \EndIf
    \EndIf
\EndFor
\For{$k = 1, \cdots, n-1$}
    \State Compute $\E{1_{k}^{(n)}\vert \mathcal{Y}_{n}} := s_k / \sum_{j=1}^{n-1} s_j$
\EndFor

\Return $\E{\tau^{(n)} \vert\mathcal{Y}_n} = \sum_{k=1}^{n-1}\E{1_{k}^{(n)}\vert \mathcal{Y}_{n}}\sum_{j=k+1}^{n}T_j$

\end{algorithmic}
\end{algorithm}

\clearpage
\section*{Code availability}
In \url{https://github.com/krzbar/YuleHeightMoments}  we provide 
code verifying our calculations. We checked the formul\ae\ of Section \ref{secYuleMoments} numerically and where possible analytically with 
\proglang{Mathematica}. We further ran simulations in \proglang{R}, and in the repository random seeds, and output for  Section \ref{secNumericalSimul} can be found.

\section*{Acknowledgments}
KB, BB are supported by an ELLIIT Call C grant. 

\bibliographystyle{plainnat}
\bibliography{PCM_Mart}

\begin{thebibliography}{11}
\providecommand{\natexlab}[1]{#1}
\providecommand{\url}[1]{\texttt{#1}}
\expandafter\ifx\csname urlstyle\endcsname\relax
  \providecommand{\doi}[1]{doi: #1}\else
  \providecommand{\doi}{doi: \begingroup \urlstyle{rm}\Url}\fi

\bibitem[Bartoszek(2014)]{KBar2014art}
K.~Bartoszek.
\newblock Quantifying the effects of anagenetic and cladogenetic evolution.
\newblock \emph{Math. Biosci.}, 254:\penalty0 42--57, 2014.

\bibitem[Bartoszek(2018)]{KBar2018art}
K.~Bartoszek.
\newblock Exact and approximate limit behaviour of the {Y}ule tree's cophenetic
  index.
\newblock \emph{Math. Biosci.}, 303:\penalty0 26--45, 2018.

\bibitem[Bartoszek(2020)]{KBar2020art}
K.~Bartoszek.
\newblock A central limit theorem for punctuated equilibrium.
\newblock \emph{Stoch. Models}, 36\penalty0 (3):\penalty0 473--517, 2020.

\bibitem[Bartoszek(2023)]{KBar2023arXivH}
K.~Bartoszek.
\newblock A miscellaneous collection of closed and asymptotic formul\ae\ for
  harmonic and quadratic harmonic sums (v2).
\newblock \emph{ArXiv e-prints}, 2023.

\bibitem[Bartoszek and Erhardsson(2021)]{KBarTErh2021art}
K.~Bartoszek and T.~Erhardsson.
\newblock Normal approximation for mixtures of normal distributions and the
  evolution of phenotypic traits.
\newblock \emph{Adv. Appl. Probab.}, 53:\penalty0 168--188, 2021.

\bibitem[Bartoszek and Sagitov(2015{\natexlab{a}})]{KBarSSag2015aart}
K.~Bartoszek and S.~Sagitov.
\newblock Phylogenetic confidence intervals for the optimal trait value.
\newblock \emph{J. Appl. Probab.}, 52\penalty0 (4):\penalty0 1115--1132,
  2015{\natexlab{a}}.

\bibitem[Bartoszek and Sagitov(2015{\natexlab{b}})]{KBarSSag2015bart}
K.~Bartoszek and S.~Sagitov.
\newblock A consistent estimator of the evolutionary rate.
\newblock \emph{J. Theor. Biol.}, 371:\penalty0 69--78, 2015{\natexlab{b}}.

\bibitem[Mir et~al.(2013)Mir, Rossell{\'o}, and Rotger]{AMirFRosLRot2013art}
A.~Mir, F.~Rossell{\'o}, and L.~Rotger.
\newblock A new balance index for phylogenetic trees.
\newblock \emph{Math. Biosci.}, 241\penalty0 (1):\penalty0 125--136, 2013.

\bibitem[Sagitov and Bartoszek(2012)]{SSagKBar2012art}
S.~Sagitov and K.~Bartoszek.
\newblock Interspecies correlation for neutrally evolving traits.
\newblock \emph{J. Theor. Biol.}, 309:\penalty0 11--19, 2012.

\bibitem[Stadler(2009)]{TreeSim1}
T.~Stadler.
\newblock On incomplete sampling under birth-death models and connections to
  the sampling-based coalescent.
\newblock \emph{J. Theor. Biol.}, 261\penalty0 (1):\penalty0 58--68, 2009.

\bibitem[Stadler(2011)]{TreeSim2}
T.~Stadler.
\newblock Simulating trees with a fixed number of extant species.
\newblock \emph{Syst. Biol.}, 60\penalty0 (5):\penalty0 676--684, 2011.

\end{thebibliography}

\section*{Appendix YA: Proofs of probabilities of coalescing at nodes from \citet{KBar2018art}}
\labeltext{Appendix YA}{appCoalProof}

\renewcommand{\theequation}{YA.\arabic{equation}}
\renewcommand{\thelemma}{YA.\arabic{lemma}}
\setcounter{equation}{0}
\setcounter{lemma}{0}
In this section we provide proofs concerning coalescent events
taking place at certain nodes. These formul\ae\  were obtained
by \citet{KBar2018art} with the help of \proglang{Mathematica}, here we present
completely analytical derivations of them.

\begin{lemma}\label{lemE1kYn2}

We have that for $n>1$ and  $1 \le k \le n-1$

$$
\begin{array}{l}
\E{\E{1_{k}^{(n)} \vert \mathcal{Y}_{n}}^{2}} = 
\frac{16n(n+1)}{(n-1)^{2}(k+1)(k+2)(k+3)(k+4)}
-\frac{16(n+1)}{(n-1)^{2}(k+1)(k+2)(k+3)}
\\+\frac{80(n+1)}{(n-1)^{2}(k+1)(k+2)(k+3)(k+4)}
+\frac{4(n+1)}{(n-1)^{2}n(k+1)(k+2)} 
-\frac{32(n+1)}{(n-1)^{2}n(k+1)(k+2)(k+3)}
\\+\frac{96(n+1)}{(n-1)^{2}n(k+1)(k+2)(k+3)(k+4)}
\end{array}
$$
\end{lemma}
\begin{proof}
We use the same approach as in \citepos{KBar2018art} 
Lemma $6.3$---considering the three possible ways that the two pairs could 
have been drawn, i.e., the same pair twice, they overlap with one tip, or are two disjoint pairs. For each of these three situations we have the expectation of the product of these indicator random variables (see below) multiplied by the probability of sampling such a pair. 
We remember that 
$$
\pi_{n,k}=\frac{2(n+1)}{n-1}\frac{1}{(k+1)(k+2)},
$$
and then

$$
\begin{array}{l}
\E{1_{k,1}^{(n)}1_{k,2}^{(n)}} = 
\binom{n}{2}^{-1}\pi_{n,k}
+2(n-2)\binom{n}{2}^{-1}4\frac{n+1}{(n-1)(n-2)}
\frac{n-(k+1)}{(k+1)(k+2)(k+3)}
\\ +\binom{n-2}{2}\binom{n}{2}^{-1}16\frac{n+1}{(n-1)(n-2)(n-3)}\frac{(n-(k+1))(n-(k+2))}{(k+1)(k+2)(k+3)(k+4)}
\\ =
\binom{n}{2}^{-1}\pi_{n,k}
\left(
1 + 4\frac{n-(k+1)}{k+3}+8\binom{n-2}{2}\frac{1}{(n-2)(n-3)}\frac{(n-(k+1))(n-(k+2))}{(k+3)(k+4)}
\right)
\\= 
\frac{4(n+1)}{n(n-1)^{2}}\frac{1}{(k+1)(k+2)}\left(
\frac{4n^{2}}{(k+3)(k+4)}
-\frac{4n(k-1)}{(k+3)(k+4)}+\frac{k^{2}-k+4}{(k+3)(k+4)}
\right)
\\=
\frac{16n(n+1)}{(n-1)^{2}(k+1)(k+2)(k+3)(k+4)}
-\frac{16(n+1)(k-1)}{(n-1)^{2}(k+1)(k+2)(k+3)(k+4)}
+\frac{4(n+1)(k^{2}-k+4)}{(n-1)^{2}n(k+1)(k+2)(k+3)(k+4)}.
\end{array}
$$
We rewrite it in a form that will be useful later

$$
\begin{array}{l}
=
\frac{16n(n+1)}{(n-1)^{2}(k+1)(k+2)(k+3)(k+4)}
-\frac{16(n+1)}{(n-1)^{2}(k+1)(k+2)(k+3)}
+\frac{80(n+1)}{(n-1)^{2}(k+1)(k+2)(k+3)(k+4)}
\\
+\frac{4(n+1)}{(n-1)^{2}n(k+1)(k+2)} 
-\frac{32(n+1)}{(n-1)^{2}n(k+1)(k+2)(k+3)}
+\frac{96(n+1)}{(n-1)^{2}n(k+1)(k+2)(k+3)(k+4)}
.
\end{array}
$$
For completeness it remains to fill in the 
parts of \citepos{KBar2018art} proof of the therein Lemma $6.3$ that  \proglang{Mathematica} was used for---$\E{1_{k,1}^{(n)}1_{k,2}^{(n)}}$ conditional on how the two random pairs were sampled, i.e., whether they are distinct or have a tip in common (see \citepos{KBar2018art} Fig. $4$).
We show that when they have a tip in common then the conditional expectation of the above product is

$$
\begin{array}{l}
\sum\limits_{j=k+1}^{n-1}\left(1-\frac{3}{\binom{n}{2}} \right)\cdots
\left(1-\frac{3}{\binom{j+2}{2}} \right)\frac{1}{\binom{j+1}{2}}
\left(1-\frac{1}{\binom{j}{2}} \right)\cdots\left(1-\frac{1}{\binom{k+2}{2}} \right)
\frac{1}{\binom{k+1}{2}}
\\ =
\sum\limits_{j=k+1}^{n-1}\left(
\frac{(n+2)(n-3)}{n(n-1)}\cdots\frac{(j+4)(j-1)}{(j+2)(j-1)}
\frac{2}{j(j+1)}\frac{(j+1)(j-2)}{j(j-1)}\cdots
\frac{(k+3)k}{(k+2)(k+1)}\frac{2}{k(k+1)}
\right)
\\=\sum\limits_{j=k+1}^{n-1}\left(
\frac{(n+2)!(n-3)!}{n!(n-1)!}\frac{(j+1)!j!}{(j+3)!(j-2)!}
\frac{2}{j(j+1)}
\frac{(j+1)!(j-2)!}{j!(j-1)!}\frac{(k+1)!k!}{(k+2)!(k-1)!}
\frac{2}{k(k+1)}
\right)
\\=\sum\limits_{j=k+1}^{n-1}\left(
\frac{(n+2)(n+1)}{(n-1)(n-2)}\frac{(j-1)j}{(j+2)(j+3)}
\frac{2}{j(j+1)}
\frac{(j+1)}{j}\frac{k}{(k+2)}
\frac{2}{k(k+1)}
\right)
\end{array}
$$
$$
\begin{array}{l}
\\=4\frac{(n+2)(n+1)}{(n-1)(n-2)}\frac{1}{(k+1)(k+2)}\sum\limits_{j=k+1}^{n-1}\frac{1}{(j+2)(j+3)}
 \stackrel{Lemma~\ref{HarXiv-lemSumj2j3}}{=}
4\frac{(n+2)(n+1)}{(n-1)(n-2)}\frac{1}{(k+1)(k+2)}\left(\frac{1}{3}-\frac{1}{n+2}\right)
\\=\frac{4(n+1)}{(n-1)(n-2)}\frac{(n-(k+1))}{(k+1)(k+2)(k+3)},
\end{array}
$$
and when the two pairs are distinct this conditional expectation becomes
(here a typo in \citepos{KBar2018art} proof of Lemma 6.3 is properly adjustes, the second, inner, sum should run up to $j_{2}-1$)

$$
\begin{array}{l}
\sum\limits_{j_{2}=k+2}^{n-1}\sum\limits_{j_{1}=k+1}^{j_{2}-1}
\left(1-\frac{6}{\binom{n}{2}} \right)\cdots
\left(1-\frac{6}{\binom{j_{2}+2}{2}} \right)\frac{4}{\binom{j_{2}+1}{2}}
\left(1-\frac{3}{\binom{j_{2}}{2}} \right)\cdots\left(1-\frac{3}{\binom{j_{1}+2}{2}} \right)
\frac{1}{\binom{j_{1}+1}{2}}
\\ \cdot \left(1-\frac{1}{\binom{j_{1}}{2}} \right)\cdots\left(1-\frac{1}{\binom{k+2}{2}} \right)\frac{1}{\binom{k+1}{2}}
\\ = 
\sum\limits_{j_{2}=k+2}^{n-1}\sum\limits_{j_{1}=k+1}^{j_{2}-1}\left(
\frac{(n+3)(n-4)}{n(n-1)}\cdots\frac{(j_{2}+5)(j_{2}-2)}{(j_{2}+2)(j_{2}+1)}
\frac{8}{j_{2}(j_{2}+1)}\frac{(j_{2}+2)(j_{2}+1)}{(j_{2}-1)(j_{2}-2)}
\frac{(j_{1}-1)j_{1}}{(j_{1}+2)(j_{1}+3)}
\frac{2}{j_{1}(j_{1}+1)}\frac{j_{1}+1}{j_{1}-1}\frac{k}{k+2}
\frac{2}{k(k+1)}
\right)
\\ = 
\frac{32(n+3)(n+2)(n+1)}{(n-1)(n-2)(n-3)(k+2)(k+1)}
\sum\limits_{j_{2}=k+2}^{n-1}\frac{1}{(j_{2}+3)(j_{2}+4)}
\sum\limits_{j_{1}=k+1}^{j_{2}-1}\frac{1}{(j_{1}+2)(j_{1}+3)}
\\ \stackrel{Lemma~\ref{HarXiv-lemSumj2j3}}{=}
\frac{32(n+3)(n+2)(n+1)}{(n-1)(n-2)(n-3)(k+2)(k+1)}\left(
\frac{1}{(k+3)}\sum\limits_{j_{2}=k+2}^{n-1}\frac{1}{(j_{2}+3)(j_{2}+4)}
-\sum\limits_{j_{2}=k+2}^{n-1}\frac{1}{(j_{2}+2)(j_{2}+3)(j_{2}+4)}
\right)
\\ \stackrel{Lemma~\ref{HarXiv-lemSumj3j4},\ref{HarXiv-lemSumj2j3j4}}{=}
\\
\frac{32(n+3)(n+2)(n+1)}{(n-1)(n-2)(n-3)(k+2)(k+1)}\left(
\frac{1}{(k+3)(k+5)}-\frac{1}{(k+3)(n+3)}-\frac{1}{2(k+4)}+\frac{1}{2(k+5)}+\frac{1}{2(n+2)}-\frac{1}{2(n+3)}
\right)
\\=
\frac{16(n+1)}{(n-1)(n-2)(n-3)}\frac{(n-(k+1))(n-(k+2))}{(k+1)(k+2)(k+3)(k+4)}.
\end{array}
$$
We can also recover $\var{\E{1_{k}^{(n)} \vert \mathcal{Y}_{n}}}$ 
(we correctly write the variance formula here as in \citepos{KBar2018art} Lemma $6.3$ it was copied with typos from \proglang{Mathematica's} output)

$$
\begin{array}{l}
 \var{\E{1_{k}^{(n)} \vert \mathcal{Y}_{n}}} =\E{\E{1_{k}^{(n)} \vert \mathcal{Y}_{n}}^{2}}   - \pi_{n,k}^{2} 
\\ =
\frac{16n(n+1)}{(n-1)^{2}(k+1)(k+2)(k+3)(k+4)} 
-\frac{16(n+1)(k-1)}{(n-1)^{2}(k+1)(k+2)(k+3)(k+4)} 
+\frac{4(n+1)(k^{2} -k + 4)}{(n-1)^{2}n(k+1)(k+2)(k+3)(k+4)}
\\
- \left(2\frac{n+1}{n-1}\frac{1}{(k+1)(k+2)} \right)^{2}
\\=
4\frac{n+1}{n(n-1)^{2}}\frac{(n-(k+1))(n(3k^{2}+5k-4)-(k^{3}+k^{2}+2k+8))}{(k+1)^{2}(k+2)^{2}(k+3)(k+4)} .
\end{array}
$$
\end{proof}
~\\~\\~\\

\begin{lemma}
For $1\le k_{1}<k_{2}\le n-1$ it holds that

\be
\begin{array}{l}
\E{\E{1_{k_{2}}^{(n)} \vert \mathcal{Y}_{n}}\E{1_{k_{1}}^{(n)} \vert \mathcal{Y}_{n}} } =
\frac{4(n+1)^{2}}{(n-1)^{2}(k_{1}+1)(k_{1}+2)(k_{2}+1)(k_{2}+2)}
\\-\frac{24n(n+1)}{(n-1)^{2}(k_{1}+1)(k_{1}+2)(k_{2}+1)(k_{2}+2)(k_{2}+3)(k_{2}+4)}
+\frac{32(n+1)}{(n-1)^{2}(k_{1}+1)(k_{1}+2)(k_{2}+1)(k_{2}+2)(k_{2}+3)}
\\-\frac{120(n+1)}{(n-1)^{2}(k_{1}+1)(k_{1}+2)(k_{2}+1)(k_{2}+2)(k_{2}+3)(k_{2}+4)}
-\frac{8(n+1)}{(n-1)^{2}n(k_{1}+1)(k_{1}+2)(k_{2}+1)(k_{2}+2)}
\\+\frac{64(n+1)}{(n-1)^{2}n(k_{1}+1)(k_{1}+2)(k_{2}+1)(k_{2}+2)(k_{2}+3)}
-\frac{144(n+1)}{(n-1)^{2}n(k_{1}+1)(k_{1}+2)(k_{2}+1)(k_{2}+2)(k_{2}+3)(k_{2}+4)}.
\end{array}
\ee
\end{lemma}
\begin{proof}
From \citepos{KBar2018art} Lemma $6.4$ we obtain

$$
\begin{array}{l}
\E{\E{1_{k_{2}}^{(n)} \vert \mathcal{Y}_{n}}\E{1_{k_{1}}^{(n)} \vert \mathcal{Y}_{n}} } =
\cov{\E{1_{k_{2}}^{(n)} \vert \mathcal{Y}_{n}}}{\E{1_{k_{1}}^{(n)} \vert \mathcal{Y}_{n}} } +
\pi_{n,k_{1}}\pi_{n,k_{2}}
\\=
\frac{-8(n+1)(3n-(k_{2}-2))(n-(k_{2}+1))}{n(n-1)^{2}(k_{1}+1)(k_{1}+2)(k_{2}+1)(k_{2}+2)(k_{2}+3)(k_{2}+4)}
+4\frac{(n+1)^{2}}{(n-1)^{2}}\frac{1}{(k_{1}+1)(k_{1}+2)(k_{2}+1)(k_{2}+2)}
\\=
\frac{4(n+1)^{2}}{(n-1)^{2}}\frac{1}{(k_{1}+1)(k_{1}+2)(k_{2}+1)(k_{2}+2)}
-\frac{24n(n+1)}{(n-1)^{2}}
\frac{1}{(k_{1}+1)(k_{1}+2)(k_{2}+1)(k_{2}+2)(k_{2}+3)(k_{2}+4)}
\\+\frac{24(n+1)}{(n-1)^{2}}\frac{1}{(k_{1}+1)(k_{1}+2)(k_{2}+1)(k_{2}+2)(k_{2}+3)}
-\frac{120(n+1)}{(n-1)^{2}}\frac{1}{(k_{1}+1)(k_{1}+2)(k_{2}+1)(k_{2}+2)(k_{2}+3)(k_{2}+4)}
\\+\frac{8(n+1)}{(n-1)^{2}}\frac{1}{(k_{1}+1)(k_{1}+2)(k_{2}+1)(k_{2}+2)(k_{2}+3)}
-\frac{8(n+1)}{n(n-1)^{2}}\frac{1}{(k_{1}+1)(k_{1}+2)(k_{2}+1)(k_{2}+2)}
\\+\frac{64(n+1)}{n(n-1)^{2}}\frac{1}{(k_{1}+1)(k_{1}+2)(k_{2}+1)(k_{2}+2)(k_{2}+3)}
-\frac{144(n+1)}{n(n-1)^{2}}\frac{1}{(k_{1}+1)(k_{1}+2)(k_{2}+1)(k_{2}+2)(k_{2}+3)(k_{2}+4)}.
\end{array}
$$
For completeness it remains to fill in the 
parts of \citepos{KBar2018art} proof of the therein Lemma $6.4$ that  \proglang{Mathematica} was used for---the value of the expectation of this product of indicator random variables, depending on how the two pairs of tips were sampled and how their coalescence pattern looks like (see \citepos{KBar2018art} Fig. $5$).
We show that when the two pairs have a tip in common,

$$
\begin{array}{l}
\left(1-\frac{3}{\binom{n}{2}} \right)\cdots
\left(1-\frac{3}{\binom{k_{2}+2}{2}} \right)\frac{1}{\binom{k_{2}+1}{2}}
\left(1-\frac{1}{\binom{k_{2}}{2}} \right)\cdots\left(1-\frac{1}{\binom{k_{1}+2}{2}} \right)
\frac{1}{\binom{k_{1}+1}{2}}
\\ =
\frac{(n+2)(n-3)}{n(n-1)}\cdots\frac{(k_{2}+4)(k_{2}-1)}{(k_{2}+2)(k_{2}-1)}
\frac{2}{k_{2}(k_{2}+1)}\frac{(k_{2}+1)(k_{2}-2)}{k_{2}(k_{2}-1)}\cdots
\frac{(k_{1}+3)k}{(k_{1}+2)(k_{1}+1)}\frac{2}{k_{1}(k_{1}+1)}
\\=
\frac{4(n+1)(n+2)}{(n-1)(n-2)}\frac{1}{(k_{1}+1)(k_{1}+2)(k_{2}+2)(k_{2}+3)},
\end{array}
$$
and when they are two disjoint pairs (here typo in \citepos{KBar2018art} proof of Lemma 6.4 is removed accordingly, in the third part of the sum it should write $1/\binom{k_{2}+1}{2}$)

$$
\begin{array}{l}
\left(1-\frac{6}{\binom{n}{2}} \right)\cdots
\left(1-\frac{6}{\binom{k_{2}+2}{2}} \right)\frac{1}{\binom{k_{2}+1}{2}}
\left(1-\frac{3}{\binom{k_{2}}{2}} \right)\cdots\left(1-\frac{3}{\binom{k_{1}+2}{2}} \right)
\frac{1}{\binom{k_{1}+1}{2}}
\\ +\sum\limits_{j=k_{1}+1}^{k_{2}-1}
\left(1-\frac{6}{\binom{n}{2}} \right)\cdots
\left(1-\frac{6}{\binom{k_{2}+2}{2}} \right)\frac{1}{\binom{k_{2}+1}{2}}
\left(1-\frac{3}{\binom{k_{2}}{2}} \right)\cdots\left(1-\frac{3}{\binom{j+2}{2}} \right)
\frac{2}{\binom{j+1}{2}}
\\ \cdot \left(1-\frac{1}{\binom{j}{2}} \right)\cdots\left(1-\frac{1}{\binom{k_{1}+2}{2}} \right)
\frac{1}{\binom{k_{1}+1}{2}}
\\ + \sum\limits_{j=k_{2}+1}^{n-1}
\left(1-\frac{6}{\binom{n}{2}} \right)\cdots
\left(1-\frac{6}{\binom{j+2}{2}} \right)\frac{4}{\binom{j+1}{2}}
\left(1-\frac{3}{\binom{j}{2}} \right)\cdots\left(1-\frac{3}{\binom{k_{2}+2}{2}} \right)
\frac{1}{\binom{k_{2}+1}{2}}
\\ \cdot \left(1-\frac{1}{\binom{k_{2}}{2}} \right)\cdots\left(1-\frac{1}{\binom{k_{1}+2}{2}} \right)
\frac{1}{\binom{k_{1}+1}{2}}
\end{array}
$$

$$
\begin{array}{l}
=
\frac{(n+3)(n+2)(n+1)}{(n-1)(n-2)(n-3)}\frac{k_{2}(k_{2}-1)(k_{2}-2)}{(k_{2}+4)(k_{2}+3)(k_{2}+2)}\frac{2}{k_{2}(k_{2}+1)}
\frac{(k_{2}+2)(k_{2}+1)}{(k_{2}-1)(k_{2}-2)}\frac{(k_{1}-1)k_{1}}{(k_{1}+2)(k_{1}+3)}\frac{2}{k_{1}(k_{1}+1)}
\\ +\sum\limits_{j=k_{1}+1}^{k_{2}-1}
\frac{(n+3)(n+2)(n+1)}{(n-1)(n-2)(n-3)}\frac{k_{2}(k_{2}-1)(k_{2}-2)}{(k_{2}+4)(k_{2}+3)(k_{2}+2)}\frac{2}{k_{2}(k_{2}+1)}
\frac{(k_{2}+2)(k_{2}+1)}{(k_{2}-1)(k_{2}-2)}\frac{(j-1)j}{(j+2)(j+3)}\frac{4}{j(j+1)}\frac{(j+1)}{(j-1)}\frac{k_{1}}{k_{1}+2}\frac{2}{k_{1}(k_{1}+1)}
\\ + \sum\limits_{j=k_{2}+1}^{n-1}
\frac{(n+3)(n+2)(n+1)}{(n-1)(n-2)(n-3)}\frac{j(j-1)(j-2)}{(j+4)(j+3)(j+2)}\frac{8}{j(j+1)}
\frac{(j+2)(j+1)}{(j-1)(j-2)}\frac{(k_{2}-1)k_{2}}{(k_{2}+2)(k_{2}+3)}\frac{2}{k_{2}(k_{2}+1)}
\frac{(k_{2}+1)}{(k_{2}-1)}\frac{k_{1}}{k_{1}+2}\frac{2}{k_{1}(k_{1}+1)}
\\=
\frac{4(n+3)(n+2)(n+1)}{(n-1)(n-2)(n-3)}\frac{1}{(k_{2}+3)(k_{2}+4)}\frac{(k_{1}-1)}{(k_{1}+1)(k_{1}+2)(k_{1}+3)}
\\ +
\frac{16(n+3)(n+2)(n+1)}{(n-1)(n-2)(n-3)}\frac{1}{(k_{2}+3)(k_{2}+4)}
\frac{1}{(k_{1}+1)(k_{1}+2)}
\sum\limits_{j=k_{1}+1}^{k_{2}-1}\frac{1}{(j+2)(j+3)}
\\ + \frac{32(n+3)(n+2)(n+1)}{(n-1)(n-2)(n-3)}
\frac{1}{(k_{2}+2)(k_{2}+3)}\frac{1}{(k_{1}+1)(k_{1}+2)}
\sum\limits_{j=k_{2}+1}^{n-1}\frac{1}{(j+3)(j+4)}
\end{array}
$$

$$
\begin{array}{l}
\stackrel{Lemmata~\ref{HarXiv-lemSumj2j3},\ref{HarXiv-lemSumj3j4} }{=}
\\
\frac{4(n+3)(n+2)(n+1)}{(n-1)(n-2)(n-3)}\frac{1}{(k_{2}+3)(k_{2}+4)}\frac{(k_{1}-1)}{(k_{1}+1)(k_{1}+2)(k_{1}+3)}
\\ +
\frac{16(n+3)(n+2)(n+1)}{(n-1)(n-2)(n-3)}\frac{1}{(k_{2}+3)(k_{2}+4)}
\frac{1}{(k_{1}+1)(k_{1}+2)}\left(\frac{1}{k_{1}+3}-\frac{1}{k_{2}+2}
\right)
\\ + \frac{32(n+3)(n+2)(n+1)}{(n-1)(n-2)(n-3)}
\frac{1}{(k_{2}+2)(k_{2}+3)}\frac{1}{(k_{1}+1)(k_{1}+2)}
\left(\frac{1}{k_{2}+4}-\frac{1}{n+3}
\right)
\\ =
\frac{4(n+1)(n+2)}{(n-1)(n-2)(n-3)}\frac{n(k_{2}+6)-5k_{2}-14}{(k_{1}+1)(k_{1}+2)(k_{2}+2)(k_{2}+3)(k_{2}+4)}
.
\end{array}
$$
\end{proof}

\section*{Appendix YB: Direct, alternative proofs of Lemmata 
\ref{lemUnmTaun2} and \ref{lemUnmETaunYn2}
}
\labeltext{Appendix YB}{appAltProof}

\noindent\textbf{Lemma \ref{lemUnmTaun2}} 
For a Yule tree with speciation rate $\lambda=1$ we have
\bd
\E{(U^{(n)}-\tau^{(n)})^{2}}= \frac{4n+2}{n-1}H_{n,2} -\frac{2H_{n,1}^{2}}{n-1}-\frac{4H_{n,1}}{n-1}
\ed
\begin{proof}
We derived Lemma \ref{lemUnmTaun2} from a joint moments formula but
the same result can be obtained directly. 
By Lemmata
\ref{HarXiv-lemSumHii1}, \ref{HarXiv-lemSumHi2i1i2}, and \ref{HarXiv-lemSumHisqi1i2app} we have
$$
\begin{array}{l}
\E{(U^{(n)}-\tau^{(n)})^{2}}= \sum\limits_{k=1}^{n-1}\pi_{n,k}\E{(U^{(n)}-\tau^{(n)})^{2}\vert \kappa_{n} = k}=
\sum\limits_{k=1}^{n-1}\pi_{n,k}\E{\left(\sum\limits_{i=1}^{k}T_{k}\right)^{2}}
\\=\sum\limits_{k=1}^{n-1}\pi_{n,k}\left(\sum\limits_{i=1}^{k}\E{T_{i}^{2}}+2\sum\limits_{i_{2}<i_{1}}^{k}\E{T_{i_{1}}T_{i_{2}}}\right)
=\frac{2(n+1)}{n-1}\sum\limits_{k=1}^{n-1}\frac{1}{(k+1)(k+2)}\left(\sum\limits_{i=1}^{k}\frac{2}{i^{2}}
+2\sum\limits_{i_{2}<i_{1}}^{k}\frac{1}{i_{1}}\frac{1}{i_{2}}\right)
\\
=\frac{2(n+1)}{n-1}\left(\sum\limits_{k=1}^{n-1}\frac{2H_{k,2}}{(k+1)(k+2)}
+2\sum\limits_{k=1}^{n-1}\frac{1}{(k+1)(k+2)}\left(\sum\limits_{i_{1}=1}^{k}\frac{H_{i_{1},1}}{i_{1}}-1-\sum\limits_{i_{1}=1}^{k}\frac{1}{i_{1}^{2}}+1\right)
\right)
\\
\stackrel{\mathrm{Lemma~ }\ref{HarXiv-lemSumHii1}}{=} 
\frac{2(n+1)}{n-1}\left(\sum\limits_{k=1}^{n-1}\frac{2H_{k,2}}{(k+1)(k+2)}
+2\sum\limits_{k=1}^{n-1}\frac{1}{(k+1)(k+2)}\left(\frac{1}{2}(H_{k,1}^{2}+H_{k,2})-H_{k,2}\right)
\right)
\\=\frac{2(n+1)}{n-1}\left(\sum\limits_{k=1}^{n-1}\frac{H_{k,2}}{(k+1)(k+2)}+\sum\limits_{k=1}^{n-1}\frac{H_{k,1}^{2}}{(k+1)(k+2)}\right)
\\ 
\stackrel{\mathrm{Lemmata~}\ref{HarXiv-lemSumHi2i1i2},~\ref{HarXiv-lemSumHisqi1i2app}}{=}
\frac{2(n+1)}{n-1}\left(\frac{n}{n+1}\left(H_{n,2}-1\right)+H_{n,2}
+\frac{(n-1)^{2}+n-1-1}{n^{2}}
- 2\frac{H_{n-1,1}}{n} -\frac{H_{n-1,1}^{2}}{n+1}\right)
\\ =
\frac{4n+2}{n-1}H_{n,2} -\frac{2H_{n,1}^{2}}{n-1}-\frac{4H_{n,1}}{n-1}
.
\end{array}
$$
We can see that we obtain the same formula as in Eq. \eqref{eqUnmTaun2}.
\end{proof}

\noindent\textbf{Lemma \ref{lemUnmETaunYn2}} 
For a Yule tree with speciation rate $\lambda=1$ we have  
\bd
\begin{array}{l}
\E{(U^{(n)}-\E{\tau^{(n)}\vert \mathcal{Y}_{n}})^{2}} = 
\frac{4}{3}H_{n,2}+3-8n^{-1}H_{n,1}+\Theta(n^{-1})
\xrightarrow{n\to \infty} \frac{2\pi^{2}}{9}+3. 
\end{array}
\ed

\begin{proof}
We present here an alternative, direct, proof
of Lemma \ref{lemUnmETaunYn2}
based on the same techniques that were used 
to show Lemmata $11$ \citep{KBarSSag2015aart} and $5.1$, $5.3$, $5.5$ \citep{KBar2020art}.
Denote by $\tau^{(n)}_{1}$ and $\tau^{(n)}_{2}$ two versions of $\tau^{(n)}$
that are independent given $\mathcal{Y}_{n}$. This allows us to write
$$
\begin{array}{l}
\E{(U^{(n)}-\E{\tau^{(n)}\vert \mathcal{Y}_{n}})^{2}} =
\E{(\E{U^{(n)}-\tau^{(n)}\vert \mathcal{Y}_{n}})^{2}} 
=
\E{(U^{(n)}-\tau^{(n)}_{1})(U^{(n)}-\tau^{(n)}_{2})}
\\=
\sum\limits_{k=1}^{n-1} \E{\left(\sum\limits_{i=1}^{k}T_{i} \right)^{2}} \E{\E{1_{k}^{(n)} \vert \mathcal{Y}_{n}}^{2}}
+2
\sum\limits_{k_{1}=1}^{n-1}
\sum\limits_{k_{2}=k_{1}+1}^{n-1}
\E{\left(\sum\limits_{i_{1}=1}^{k_{1}}T_{i_{1}} \right)\left(\sum\limits_{i_{2}=1}^{k_{2}}T_{i_{2}} \right)} \E{1_{k,1}^{(n)}1_{k,2}^{(n)}}
\end{array}
$$
$$
\begin{array}{l}
=
\sum\limits_{k=1}^{n-1} \E{\left(\sum\limits_{i=1}^{k}T_{i} \right)^{2}} \E{\E{1_{k}^{(n)} \vert \mathcal{Y}_{n}}^{2}}
\\+2\left(
\sum\limits_{k_{1}=1}^{n-1}
\sum\limits_{k_{2}=k_{1}+1}^{n-1}
\E{\left(\sum\limits_{i_{1}=1}^{k_{1}}T_{i_{1}} \right)\left(\sum\limits_{i_{2}=1}^{k_{1}}T_{i_{2}} +  
\sum\limits_{i_{2}=k_{1}+1}^{k_{2}}T_{i_{2}}
\right)} \E{1_{k,1}^{(n)}1_{k,2}^{(n)}}
\right)
\\=
\sum\limits_{k=1}^{n-1} \E{\left(\sum\limits_{i=1}^{k}T_{i} \right)^{2}} \E{\E{1_{k}^{(n)} \vert \mathcal{Y}_{n}}^{2}}
\\+2\left(
\sum\limits_{k_{1}=1}^{n-1}
\sum\limits_{k_{2}=k_{1}+1}^{n-1}
\left(\E{\left(\sum\limits_{i_{1}=1}^{k_{1}}T_{i_{1}} \right)^{2}}
 +  
\E{\left(\sum\limits_{i_{1}=1}^{k_{1}}T_{i_{1}} \right) }
\E{\left(\sum\limits_{i_{2}=k_{1}+1}^{k_{2}}T_{i_{2}}
\right)} \right)\E{1_{k,1}^{(n)}1_{k,2}^{(n)}}
\right).
\end{array}
$$
Using that $T_{i}$ is exponentially distributed with rate $i$, we
calculate 

$$
\begin{array}{l}
E{\left(\sum\limits_{i=1}^{k}T_{i} \right)^{2}} 
=
\E{\sum\limits_{i=1}^{k}T_{i}^{2}+
2\sum\limits_{i_{1}=1}^{k}\sum\limits_{i_{2}=i_{1}+1}^{k}T_{i_{1}}T_{i_{2}}}
\\=
\sum\limits_{i=1}^{k}\E{T_{i}^{2}}+
2\sum\limits_{i_{1}=1}^{k}\sum\limits_{i_{2}=i_{1}+1}^{k}\E{T_{i_{1}}}\E{T_{i_{2}}}
=
2H_{k,2}+
2\sum\limits_{i_{1}=1}^{k}\frac{1}{i_{1}}\left(H_{k,1}-H_{i_{1},1}\right)
\\=
2H_{k,2}+
2H_{k,1}^{2}-2\sum\limits_{i_{1}=1}^{k}\frac{H_{i_{1},1}}{i_{1}}
\stackrel{Lemma~ \ref{HarXiv-lemSumHii1}}{=}
2H_{k,2}+
2H_{k,1}^{2}-\left(H_{k,1}^{2}+H_{k,2}\right)=
H_{k,2}+H_{k,1}^{2}.
\end{array} 
$$
We continue with plugging the above calculation back into our main derivations 

$$
\begin{array}{l}
\E{(U^{(n)}-\E{\tau^{(n)}\vert \mathcal{Y}_{n}})^{2}} =
\sum\limits_{k=1}^{n-1} \left(H_{k,2}+H_{k,1}^{2}\right)
\E{\E{1_{k}^{(n)} \vert \mathcal{Y}_{n}}^{2}}
\\+2\left(
\sum\limits_{k_{1}=1}^{n-1}
\sum\limits_{k_{2}=k_{1}+1}^{n-1}
\left(\left(H_{k_{1},2}+H_{k_{1},1}^{2}\right)
 +  H_{k_{1},1}\left(H_{k_{2},1}-H_{k_{1},1} \right)
 \right)\E{1_{k,1}^{(n)}1_{k,2}^{(n)}}
\right)
\\ = 
\sum\limits_{k=1}^{n-1} \left(H_{k,2}+H_{k,1}^{2}\right)
\E{\E{1_{k}^{(n)} \vert \mathcal{Y}_{n}}^{2}}
\\ +2\left(
\sum\limits_{k_{1}=1}^{n-1}
\sum\limits_{k_{2}=k_{1}+1}^{n-1}
\left(H_{k_{1},2}+H_{k_{1},1}^{2} +  H_{k_{1},1}H_{k_{2},1}-H_{k_{1},1}^{2}
 \right)\E{1_{k,1}^{(n)}1_{k,2}^{(n)}}
\right)
\end{array}
$$
$$
\begin{array}{l}
= 
\sum\limits_{k=1}^{n-1} \left(H_{k,2}+H_{k,1}^{2}\right)
\left(
\frac{16n(n+1)}{(n-1)^{2}(k+1)(k+2)(k+3)(k+4)}
-\frac{16(n+1)}{(n-1)^{2}(k+1)(k+2)(k+3)}
\right. \\ \left.
+\frac{80(n+1)}{(n-1)^{2}(k+1)(k+2)(k+3)(k+4)}
+\frac{4(n+1)}{(n-1)^{2}n(k+1)(k+2)} 
-\frac{32(n+1)}{(n-1)^{2}n(k+1)(k+2)(k+3)}
\right. \\ \left.
+\frac{96(n+1)}{(n-1)^{2}n(k+1)(k+2)(k+3)(k+4)}
\right)
\\ +2\left(
\sum\limits_{k_{1}=1}^{n-1}
\sum\limits_{k_{2}=k_{1}+1}^{n-1}
\left(H_{k_{1},2} +  H_{k_{1},1}H_{k_{2},1}
 \right)
 \left(
\frac{4(n+1)^{2}}{(n-1)^{2}(k_{1}+1)(k_{1}+2)(k_{2}+1)(k_{2}+2)}
\right. \right. \\ \left. \left. 
-\frac{24n(n+1)}{(n-1)^{2}(k_{1}+1)(k_{1}+2)(k_{2}+1)(k_{2}+2)(k_{2}+3)(k_{2}+4)}
+\frac{32(n+1)}{(n-1)^{2}(k_{1}+1)(k_{1}+2)(k_{2}+1)(k_{2}+2)(k_{2}+3)}
\right. \right. \\ \left. \left.
-\frac{120(n+1)}{(n-1)^{2}(k_{1}+1)(k_{1}+2)(k_{2}+1)(k_{2}+2)(k_{2}+3)(k_{2}+4)}
-\frac{8(n+1)}{(n-1)^{2}n(k_{1}+1)(k_{1}+2)(k_{2}+1)(k_{2}+2)}
\right. \right. \\ \left. \left.
+\frac{64(n+1)}{(n-1)^{2}n(k_{1}+1)(k_{1}+2)(k_{2}+1)(k_{2}+2)(k_{2}+3)}
-\frac{144(n+1)}{(n-1)^{2}n(k_{1}+1)(k_{1}+2)(k_{2}+1)(k_{2}+2)(k_{2}+3)(k_{2}+4)}
 \right)
\right)
\end{array}
$$

$$
\begin{array}{l}
 = 
\left(\frac{16n(n+1)}{(n-1)^{2}}+\frac{80(n+1)}{(n-1)^{2}}+\frac{96(n+1)}{(n-1)^{2}n}\right)
\sum\limits_{k=1}^{n-1}\frac{H_{k,2}+H_{k,1}^{2}}{(k+1)(k+2)(k+3)(k+4)}
\\-\left(\frac{16(n+1)}{(n-1)^{2}}+\frac{32(n+1)}{(n-1)^{2}n}\right)
\sum\limits_{k=1}^{n-1}\frac{H_{k,2}+H_{k,1}^{2}}{(k+1)(k+2)(k+3)}
 +\frac{4(n+1)}{(n-1)^{2}n} 
\sum\limits_{k=1}^{n-1}\frac{H_{k,2}+H_{k,1}^{2}}{(k+1)(k+2)}
\\ 
+2\sum\limits_{k_{1}=1}^{n-1}\frac{H_{k_{1},2}}{(k_{1}+1)(k_{1}+2)}\left(
\left(\frac{4(n+1)^{2}}{(n-1)^{2}}-\frac{8(n+1)}{(n-1)^{2}n} \right)
\sum\limits_{k_{2}=k_{1}+1}^{n-1}\frac{1}{(k_{2}+1)(k_{2}+2)}
\right. \\ \left. 
+\left(\frac{32(n+1)}{(n-1)^{2}}+\frac{64(n+1)}{(n-1)^{2}n}\right)
\sum\limits_{k_{2}=k_{1}+1}^{n-1}\frac{1}{(k_{2}+1)(k_{2}+2)(k_{2}+3)}
\right. \\ \left. 
-\left(\frac{24n(n+1)}{(n-1)^{2}}+\frac{120(n+1)}{(n-1)^{2}}+\frac{144(n+1)}{(n-1)^{2}n}\right)
\sum\limits_{k_{2}=k_{1}+1}^{n-1}\frac{1}{(k_{2}+1)(k_{2}+2)(k_{2}+3)(k_{2}+4)}
\right)
\\
+2\sum\limits_{k_{1}=1}^{n-1}\frac{H_{k_{1},1}}{(k_{1}+1)(k_{1}+2)}
\left(
\left(\frac{4(n+1)^{2}}{(n-1)^{2}}-\frac{8(n+1)}{(n-1)^{2}n} \right)
\sum\limits_{k_{2}=k_{1}+1}^{n-1}\frac{H_{k_{2},1}}{(k_{2}+1)(k_{2}+2)}
\right. \\ \left. 
+\left(\frac{32(n+1)}{(n-1)^{2}}+\frac{64(n+1)}{(n-1)^{2}n}\right)
\sum\limits_{k_{2}=k_{1}+1}^{n-1}\frac{H_{k_{2},1}}{(k_{2}+1)(k_{2}+2)(k_{2}+3)}
\right. \\ \left. 
-\left(\frac{24n(n+1)}{(n-1)^{2}}+\frac{120(n+1)}{(n-1)^{2}}+\frac{144(n+1)}{(n-1)^{2}n}\right)
\sum\limits_{k_{2}=k_{1}+1}^{n-1}\frac{H_{k_{2},1}}{(k_{2}+1)(k_{2}+2)(k_{2}+3)(k_{2}+4)}
\right)
\end{array}
$$
$$
\begin{array}{l}
\stackrel{Lemmata~ \ref{HarXiv-lemSum1},\ref{HarXiv-lemSumj1j2j3},\ref{HarXiv-lemSumj1j2j3j4},\ref{HarXiv-lemSumHii1i2},\ref{HarXiv-lemSumHj1j1j2j3},\ref{HarXiv-lemSumHj1j1j2j3j4}}{=}
\\ \left(\frac{16n(n+1)}{(n-1)^{2}}+\frac{80(n+1)}{(n-1)^{2}}+\frac{96(n+1)}{(n-1)^{2}n}\right)
\sum\limits_{k=1}^{n-1}\frac{H_{k,2}+H_{k,1}^{2}}{(k+1)(k+2)(k+3)(k+4)}
\\-\left(\frac{16(n+1)}{(n-1)^{2}}+\frac{32(n+1)}{(n-1)^{2}n}\right)
\sum\limits_{k=1}^{n-1}\frac{H_{k,2}+H_{k,1}^{2}}{(k+1)(k+2)(k+3)}
 +\frac{4(n+1)}{(n-1)^{2}n} 
\sum\limits_{k=1}^{n-1}\frac{H_{k,2}+H_{k,1}^{2}}{(k+1)(k+2)}
\\ 
+2\sum\limits_{k_{1}=1}^{n-1}\frac{H_{k_{1},2}}{(k_{1}+1)(k_{1}+2)}\left(
\left(\frac{4(n+1)^{2}}{(n-1)^{2}}-\frac{8(n+1)}{(n-1)^{2}n} \right)
\left(\frac{1}{k_{1}+2}-\frac{1}{n+1}\right)
\right. \\ \left. 
+\left(\frac{32(n+1)}{(n-1)^{2}}+\frac{64(n+1)}{(n-1)^{2}n}\right)
\left(\frac{1}{2(k_{1}+2)}-\frac{1}{2(k_{1}+3)}-\frac{1}{2(n+1)}+\frac{1}{2(n+2)} \right)
\right. \\ \left. 
-\left(\frac{24n(n+1)}{(n-1)^{2}}+\frac{120(n+1)}{(n-1)^{2}}+\frac{144(n+1)}{(n-1)^{2}n}\right)
\left(
\frac{1}{6(k_{1}+2)}-\frac{1}{3(k_{1}+3)}+\frac{1}{6(k_{1}+4)}
-\frac{1}{6(n+1)}+\frac{1}{3(n+2)}-\frac{1}{6(n+3)}
\right)
\right)
\\
+2\sum\limits_{k_{1}=1}^{n-1}\frac{H_{k_{1},1}}{(k_{1}+1)(k_{1}+2)}
\left(
\left(\frac{4(n+1)^{2}}{(n-1)^{2}}-\frac{8(n+1)}{(n-1)^{2}n} \right)
\left(
\frac{H_{k_{1},1}}{k_{1}+2}+\frac{1}{(k_{1}+1)(k_{1}+2)}+\frac{1}{k_{1}+2}
-\frac{H_{n,1}}{n+1}-\frac{1}{n+1}
\right)
\right. \\ \left. 
+\left(\frac{32(n+1)}{(n-1)^{2}}+\frac{64(n+1)}{(n-1)^{2}n}\right)
\left(
\frac{H_{k_{1},1}}{2(k_{1}+2)}
-\frac{H_{k_{1},1}}{2(k_{1}+3)}
+\frac{1}{4(k_{1}+2)}-\frac{1}{4(k_{1}+3)}
+\frac{1}{2(k_{1}+1)(k_{1}+2)}
-\frac{1}{2(k_{1}+1)(k_{1}+3)}
\right.\right. \\ \left.  \left. 
-\frac{H_{n,1}}{2(n+1)}+\frac{H_{n,1}}{2(n+2)}-\frac{1}{4(n+1)}+\frac{1}{4(n+2)}
\right)
\right. \\ \left. 
-\left(\frac{24n(n+1)}{(n-1)^{2}}+\frac{120(n+1)}{(n-1)^{2}}+\frac{144(n+1)}{(n-1)^{2}n}\right)
\left(
\frac{H_{k_{1},1}}{6(k_{1}+2)}-\frac{H_{k_{1},1}}{3(k_{1}+3)}+\frac{H_{k_{1},1}}{6(k_{1}+4)}
+\frac{1}{18(k_{1}+2)}-\frac{1}{9(k_{1}+3)}+\frac{1}{18(k_{1}+4)}
\right.\right. \\ \left.  \left. 
+\frac{1}{6(k_{1}+1)(k_{1}+2)}-\frac{1}{3(k_{1}+1)(k_{1}+3)}+\frac{1}{6(k_{1}+1)(k_{1}+4)}
-\frac{H_{n,1}}{6(n+1)}+\frac{H_{n,1}}{3(n+2)}-\frac{H_{n,1}}{6(n+3)}
-\frac{1}{18(n+1)}+\frac{1}{9(n+2)}-\frac{1}{18(n+3)}
\right)
\right)
\end{array}
$$

$$
\begin{array}{l}
=
\left(\frac{-4(2n^{2}+n-1) }{(n-1)^{2}n}\right)
\sum\limits_{k_{1}=1}^{n-1}\frac{H_{k_{1},2}}{(k_{1}+1)(k_{1}+2)}
+0\cdot\sum\limits_{k_{1}=1}^{n-1}\frac{H_{k_{1},2}}{(k_{1}+1)(k_{1}+2)^{2}}
+\left(\frac{16(n+1)(n+2)}{(n-1)^{2}}\right)
\sum\limits_{k_{1}=1}^{n-1}\frac{H_{k_{1},2}}{(k_{1}+1)(k_{1}+2)(k_{1}+3)}
\\ +\left(\frac{-8(n+1)(n+2)(n+3) }{(n-1)^{2}n}\right)
\sum\limits_{k_{1}=1}^{n-1}\frac{H_{k_{1},2}}{(k_{1}+1)(k_{1}+2)(k_{1}+4)}
+\left(\frac{16(n+1)(n+2)(n+3) }{(n-1)^{2}n}\right)
\sum\limits_{k_{1}=1}^{n-1}\frac{H_{k_{1},2}}{(k_{1}+1)(k_{1}+2)(k_{1}+3)(k_{1}+4)}
\\ +\frac{4(n+1)}{(n-1)^{2}n} 
\sum\limits_{k=1}^{n-1}\frac{H_{k,1}^{2}}{(k+1)(k+2)}
+0\cdot\sum\limits_{k_{1}=1}^{n-1}\frac{H_{k_{1},1}^{2}}{(k_{1}+1)(k_{1}+2)^{2}}
+\left(\frac{16(n+1)(n+2)}{(n-1)^{2}}\right)
\sum\limits_{k_{1}=1}^{n-1}\frac{H_{k_{1},1}^{2}}{(k_{1}+1)(k_{1}+2)(k_{1}+3)}
\\ +\left(\frac{-8(n+1)(n+2)(n+3) }{(n-1)^{2}n}\right)
\sum\limits_{k_{1}=1}^{n-1}\frac{H_{k_{1},1}^{2}}{(k_{1}+1)(k_{1}+2)(k_{1}+4)}
+\left(\frac{16(n+1)(n+2)(n+3) }{(n-1)^{2}n}\right)
\sum\limits_{k_{1}=1}^{n-1}\frac{H_{k_{1},1}^{2}}{(k_{1}+1)(k_{1}+2)(k_{1}+3)(k_{1}+4)}
\\ +0\cdot\sum\limits_{k_{1}=1}^{n-1}\frac{H_{k_{1},1}}{(k_{1}+1)^{2}(k_{1}+2)^{2}}
+\left(\frac{16(n+1)(n+2)}{3(n-1)^{2}}\right) 
\sum\limits_{k_{1}=1}^{n-1}\frac{H_{k_{1},1}}{(k_{1}+1)(k_{1}+2)^{2}}
\\ +\left(\frac{-8(3n(n+1)H_{n,1}+3n(n+1)-2)}{3(n-1)^{2}n}\right)
\sum\limits_{k_{1}=1}^{n-1}\frac{H_{k_{1},1}}{(k_{1}+1)(k_{1}+2)}
+\left(\frac{16(n+1)(n+2)}{3(n-1)^{2}}\right)
\sum\limits_{k_{1}=1}^{n-1}\frac{H_{k_{1},1}}{(k_{1}+1)(k_{1}+2)(k_{1}+3)}
\\ +\left(\frac{16(n+1)^{2}(n+2)}{(n-1)^{2}n}\right)
\sum\limits_{k_{1}=1}^{n-1}\frac{H_{k_{1},1}}{(k_{1}+1)^{2}(k_{1}+2)(k_{1}+3)}
+\left(\frac{-8(n+1)(n+2)(n+3)}{3(n-1)^{2}n}\right)
\sum\limits_{k_{1}=1}^{n-1}\frac{H_{k_{1},1}}{(k_{1}+1)(k_{1}+2)(k_{1}+4)}
\\ +\left(\frac{-8(n+1)(n+2)(n+3)}{(n-1)^{2}n}\right)
\sum\limits_{k_{1}=1}^{n-1}\frac{H_{k_{1},1}}{(k_{1}+1)^{2}(k_{1}+2)(k_{1}+4)}
\end{array}
$$
$$
\begin{array}{l}
=
\left(\frac{-4(2n^{2}+n-1) }{(n-1)^{2}n}\right)
\sum\limits_{k_{1}=1}^{n-1}\frac{H_{k_{1},2}}{(k_{1}+1)(k_{1}+2)}

+\left(\frac{16(n+1)(n+2)(2n+3)}{(n-1)^{2}n}\right)
\sum\limits_{k_{1}=1}^{n-1}\frac{H_{k_{1},2}}{(k_{1}+1)(k_{1}+2)(k_{1}+3)} 

\\ +\left(\frac{-24(n+1)(n+2)(n+3)}{(n-1)^{2}n}\right)
\sum\limits_{k_{1}=1}^{n-1}\frac{H_{k_{1},2}}{(k_{1}+1)(k_{1}+2)(k_{1}+4)} 

\\ +\frac{4(n+1)}{(n-1)^{2}n} 
\sum\limits_{k=1}^{n-1}\frac{H_{k,1}^{2}}{(k+1)(k+2)}

+\left(\frac{16(n+1)(n+2)(2n+3)}{(n-1)^{2}n}\right)
\sum\limits_{k_{1}=1}^{n-1}\frac{H_{k_{1},1}^{2}}{(k_{1}+1)(k_{1}+2)(k_{1}+3)} 

\\ +\left(\frac{-24(n+1)(n+2)(n+3)}{(n-1)^{2}n}\right)
\sum\limits_{k_{1}=1}^{n-1}\frac{H_{k_{1},1}^{2}}{(k_{1}+1)(k_{1}+2)(k_{1}+4)} 

\\ +\left(\frac{16(n+1)(n+2)}{3(n-1)^{2}}\right) 
\sum\limits_{k_{1}=1}^{n-1}\frac{H_{k_{1},1}}{(k_{1}+1)(k_{1}+2)^{2}}

 +\left(\frac{-8(3n(n+1)H_{n,1}+3n(n+1)-2)}{3(n-1)^{2}n}\right)
\sum\limits_{k_{1}=1}^{n-1}\frac{H_{k_{1},1}}{(k_{1}+1)(k_{1}+2)}

\\+\left(\frac{16(n+1)(n+2)}{3(n-1)^{2}}\right)
\sum\limits_{k_{1}=1}^{n-1}\frac{H_{k_{1},1}}{(k_{1}+1)(k_{1}+2)(k_{1}+3)}

 +\left(\frac{16(n+1)^{2}(n+2)}{(n-1)^{2}n}\right)
\sum\limits_{k_{1}=1}^{n-1}\frac{H_{k_{1},1}}{(k_{1}+1)^{2}(k_{1}+2)(k_{1}+3)}

\\+\left(\frac{-8(n+1)(n+2)(n+3)}{3(n-1)^{2}n}\right)
\sum\limits_{k_{1}=1}^{n-1}\frac{H_{k_{1},1}}{(k_{1}+1)(k_{1}+2)(k_{1}+4)}

 +\left(\frac{-8(n+1)(n+2)(n+3)}{(n-1)^{2}n}\right)
\sum\limits_{k_{1}=1}^{n-1}\frac{H_{k_{1},1}}{(k_{1}+1)^{2}(k_{1}+2)(k_{1}+4)}
\end{array}
$$
$$
\begin{array}{l}
\stackrel{Lemmata~\ref{HarXiv-lemSumHii1i2}, \ref{HarXiv-lemSumHii1i2sq}, 
\ref{HarXiv-lemSumHj1j1j2j3}, \ref{HarXiv-lemSumHj1j1j2j4},\ref{HarXiv-lemSumHj1j1sqj2j3}, 
\ref{HarXiv-lemSumHj1j1sqj2j4}, \ref{HarXiv-lemSumHi2i1i2}, \ref{HarXiv-lemSumHj2j1j3},
\ref{HarXiv-lemSumHj2j1j2j4}, \ref{HarXiv-lemSumHisqi1i2app}, \ref{HarXiv-lemSumHj1sqj1j2j3}, 
\ref{HarXiv-lemSumHj1sqj1j2j4}
}{=}
\\
=
\left(\frac{-4(2n^{2}+n-1) }{(n-1)^{2}n}\right)
\left(\frac{n}{n+1}(H_{n,2}-1) \right)
\\
+\left(\frac{16(n+1)(n+2)(2n+3)}{(n-1)^{2}n}\right)\left(\frac{1}{4}H_{n,2}-\frac{5}{16}-\frac{H_{n,2}}{2(n+1)}+\frac{H_{n,2}}{2(n+2)}+\frac{3}{8(n+1)}-\frac{1}{8(n+2)}\right)
\\
+\left(\frac{-24(n+1)(n+2)(n+3)}{(n-1)^{2}n}\right)
\left(\frac{7}{36}H_{n,2}-\frac{307}{1296}-\frac{H_{n,2}}{3(n+1)}+\frac{H_{n,2}}{6(n+2)}+\frac{H_{n,2}}{6(n+3)} +\frac{59}{216(n+1)}-\frac{13}{216(n+2)}-\frac{1}{54(n+3)}\right)
\\
 +\frac{4(n+1)}{(n-1)^{2}n} \left(H_{n,2}+\frac{(n-1)^{2}+n-2}{n^{2}}-\frac{2H_{n-1,1}}{n}-\frac{H_{n-1,1}^{2}}{n+1}\right)
\\
+\left(\frac{16(n+1)(n+2)(2n+3)}{(n-1)^{2}n}\right)
\left(\frac{1}{4}H_{n,2}- \frac{3}{16}-\frac{H_{n,1}^{2}}{2(n+1)}+\frac{H_{n,1}^{2}}{2(n+2)}-\frac{H_{n,1}}{2(n+1)} +\frac{H_{n,1}}{2(n+2)}+\frac{1}{8(n+1)}+\frac{1}{8(n+2)}\right)
\\
 +\left(\frac{-24(n+1)(n+2)(n+3)}{(n-1)^{2}n}\right)
\left(\frac{7}{36}H_{n,2} - \frac{161}{1296}-\frac{H_{n,1}^{2}}{3(n+1)}+\frac{H_{n,1}^{2}}{6(n+2)}+\frac{H_{n,1}^{2}}{6(n+3)}-\frac{7H_{n,1}}{18(n+1)}+\frac{5H_{n,1}}{18(n+2)}+\frac{H_{n,1}}{9(n+3)}
\right. \\ \left.
+\frac{13}{216(n+1)}+\frac{25}{216(n+2)} + \frac{1}{54(n+3)}\right)
\\
 +\left(\frac{16(n+1)(n+2)}{3(n-1)^{2}}\right) 
\left(3-H_{n+1,2}+H_{n+1,3}-\sum\limits_{j=1}^{n+1}\frac{H_{j,1}}{j^{2}}-\frac{H_{n,1}}{n+1}-\frac{2}{n+1}\right)
\\
+\left(\frac{-8(3n(n+1)H_{n,1}+3n(n+1)-2)}{3(n-1)^{2}n}\right)
\left(1-\frac{H_{n,1}}{n+1}-\frac{1}{n+1}\right)
\\
+\left(\frac{16(n+1)(n+2)}{3(n-1)^{2}}\right)
\left(\frac{1}{8}-\frac{H_{n,1}}{2(n+1)}+\frac{H_{n,1}}{2(n+2)}-\frac{1}{4(n+1)}+\frac{1}{4(n+2)}\right)
\end{array}
$$

$$
\begin{array}{l}
+\left(\frac{16(n+1)^{2}(n+2)}{(n-1)^{2}n}\right)
\left(\frac{1}{2}\sum\limits_{j=1}^{n}\frac{H_{j,1}}{j^{2}}-\frac{1}{2}H_{n,3}-\frac{9}{16}+\frac{3H_{n+1,1}}{4(n+1)}-\frac{H_{n+2,1}}{4(n+2)}+\frac{7}{8(n+1)}-\frac{3}{8(n+2)}-\frac{3}{4(n+1)^{2}}+\frac{1}{4(n+2)^{2}}\right)
\\
+\left(\frac{-8(n+1)(n+2)(n+3)}{3(n-1)^{2}n}\right)
\left(\frac{23}{216}-\frac{H_{n+1,1}}{3(n+1)}+\frac{H_{n+2,1}}{6(n+2)}+\frac{H_{n+3,1}}{6(n+3)}-
\frac{4}{9(n+1)}+\frac{5}{36(n+2)}+\frac{11}{36(n+3)}
\right. \\ \left.
+\frac{1}{3(n+1)^{2}}-\frac{1}{6(n+2)^{2}}-\frac{1}{6(n+3)^{2}}\right)
\\
+\left(\frac{-8(n+1)(n+2)(n+3)}{(n-1)^{2}n}\right)
\left(\frac{1}{3}\sum\limits_{j=1}^{n}\frac{H_{j,1}}{j^{2}}-\frac{1}{3}H_{n,3}-\frac{239}{648}+\frac{4H_{n+1,1}}{9(n+1)}-\frac{H_{n+2,1}}{18(n+2)}-\frac{H_{n+3,1}}{18(n+3)}
\right. \\ \left. 
+\frac{13}{27(n+1)}-\frac{5}{108(n+2)}-\frac{11}{108(n+3)}-\frac{4}{9(n+1)^{2}}+\frac{1}{18(n+2)^{2}}+\frac{1}{18(n+3)^{2}}\right)
\end{array}
$$

$$
\begin{array}{l}
=
0\cdot\sum\limits_{j=1}^{n}\frac{H_{j,1}}{j^{2}}
+0\cdot H_{n,3}
+\frac{4(n^{2}-6n-4)H_{n,2}}{3(n-1)^{2}}  
+\frac{9n^{3}+30n^{2}+27n+22}{3(n-1)^{2}n}
+\frac{4H_{n,1}^{2}}{(n-1)^{2}}
+\frac{-8(3n^{2}-1)H_{n,1}}{3(n-1)^{2}n}
+\frac{2(2n^{2}-7n-11)}{3(n-1)^{2}n}
\\=
\frac{4(n^{2}-6n-4)H_{n,2}}{3(n-1)^{2}}   
+\frac{9n^{3}+30n^{2}+27n+22}{3(n-1)^{2}n} 
+\frac{4H_{n,1}^{2}}{(n-1)^{2}}
+\frac{-8(3n^{2}-1)H_{n,1}}{3(n-1)^{2}n}
+\frac{2(2n^{2}-7n-11)}{3(n-1)^{2}n}
\\=
\frac{4(n^{2}-6n-4)nH_{n,2}
+9n^{3}+30n^{2}+27n+22
+12nH_{n,1}^{2}-8(3n^{2}-1)H_{n,1}+2(2n^{2}-7n-11)
}{3(n-1)^{2}n}
\\=
\frac{
4n^{3}H_{n,2}-24n^{2}H_{n,2}-16nH_{n,2}+9n^{3}+30n^{2}+27n+22+12nH_{n,1}^{2}-24n^{2}H_{n,1}+8H_{n,1}+4n^{2}-14n-22
}{3(n-1)^{2}n}
\\=
\frac{
4n^{3}H_{n,2}+9n^{3}
-24n^{2}H_{n,1}
-24n^{2}H_{n,2}+34n^{2}
+12nH_{n,1}^{2}-16nH_{n,2}+13n
+8H_{n,1}}{3(n-1)^{2}n}
\end{array}
$$
We can see that we obtain the same formula as in Eq. \eqref{eqEUnmtaun2}.
\end{proof}

\end{document}